\documentclass[12pt, reqno]{amsart}
\usepackage[romanian,english]{babel}

\usepackage{amsfonts, amsthm, amsmath, amssymb}
\usepackage{hyperref}
\hypersetup{colorlinks=false}

\usepackage[margin=3.2cm]{geometry}

\RequirePackage{mathrsfs} \let\mathcal\mathscr
\numberwithin{equation}{section}

\renewcommand{\d}{\mathrm{d}}
\renewcommand{\phi}{\varphi}
\renewcommand{\rho}{\varrho}

\newcommand{\0}{\mathbf{0}}
\newcommand{\PP}{\mathbb{P}}
\renewcommand{\AA}{\mathbb{A}}

\newcommand{\FF}{\mathbb{F}}
\newcommand{\ZZ}{\mathbb{Z}}

\newcommand{\NN}{\mathbb{N}}
\newcommand{\QQ}{\mathbb{Q}}
\newcommand{\RR}{\mathbb{R}}
\newcommand{\CC}{\mathbb{C}}

\newcommand{\TT}{\mathbb{T}}
\newcommand{\KK}{\mathbb{K}}

\newcommand{\cO}{\mathcal{O}}

\DeclareMathOperator{\tr}{Tr}

\renewcommand{\leq}{\leqslant}

\renewcommand{\geq}{\geqslant}

\renewcommand{\bar}{\overline}

\newcommand{\x}{\mathbf{x}}
\newcommand{\y}{\mathbf{y}}

\newcommand{\z}{\mathbf{z}}

\renewcommand{\k}{\mathbf{k}}
\newcommand{\s}{\mathbf{s}}

\newcommand{\ee}{\mathbf{e}}

\newcommand{\bal}{\boldsymbol{\alpha}}

\def \Mor {\operatorname{Mor}}

\renewcommand{\t}{\mathbf{t}}

\renewcommand{\hat}{\widehat}

\newcommand{\g}{\mathbf{g}}

\newtheorem{theorem}{Theorem}[section]
\newtheorem{lemma}[theorem]{Lemma}
\newtheorem{proposition}[theorem]{Proposition}

\newtheorem{corollary}[theorem]{Corollary}

\theoremstyle{definition}
\newtheorem{hyp}[theorem]{Hypothesis}

\newtheorem{remark}[theorem]{Remark}

\newtheorem*{ack}{Acknowledgements}
\newtheorem*{summary}{Summary of the contents}

\DeclareMathOperator{\ord}{ord}

\usepackage[usenames,dvipsnames]{color}

\title[Rational surfaces]{Rational surfaces on low degree hypersurfaces}

\author{Tim Browning}
\address{IST Austria\\
Am Campus 1\\
3400 Klosterneuburg\\
Austria}
\email{tdb@ist.ac.at}

\author{Shuntaro Yamagishi}
\address{IST Austria\\
Am Campus 1\\
3400 Klosterneuburg\\
Austria}
\email{shuntaro.yamagishi@ist.ac.at}

\makeatletter
\@namedef{subjclassname@2020}{\textup{2020} Mathematics Subject Classification}
\makeatother

\date{\today}

\begin{document}

\begin{abstract}
We use function field analytic number theory to establish the irreducibility and dimension of the
moduli space that parameterises morphisms $\PP^2\to X$ of fixed degree, for an arbitrary smooth  hypersurface
$X$ of sufficiently small degree.
\end{abstract}

\maketitle
\thispagestyle{empty}

\setcounter{tocdepth}{1}
\tableofcontents

\section{Introduction}\label{s:intro}

The study of rational curves on hypersurfaces has led to major advances in algebraic geometry, with deep connections to arithmetic geometry, enumerative geometry and moduli theory. In contrast, comparatively little is known about  {\em rational surfaces} contained in hypersurfaces. In this paper we investigate the moduli space of degree $e$ morphisms $\PP^2 \to X$, where $X \subset \PP^{n-1}$ is a smooth hypersurface of degree $d \geq 2$ defined over an algebraically closed field $K$ whose characteristic exceeds $d$ if it is positive. Our main result establishes that this moduli space is irreducible and has the expected dimension  when $n$ is sufficiently large in terms of $d$ and $e$, extending techniques originally developed for the study of rational curves.

The geometry of the moduli space $\mathcal M_{0,0}(X,e)$ of degree $e$ rational curves on $X$
has been  extensively studied. 
The irreducibility and dimension of $\mathcal{M}_{0,0}(X,e)$ have been established by Kim and Pandharipande  \cite[Cor.~1]{kim} when $d=2$, and by  Coskun and Starr \cite{CS} when $d = 3$. For general hypersurfaces $X$ of degree $d$, the strongest result is due to Riedl and Yang \cite{RY}, who use a version of  Bend-and-Break to prove that
$\mathcal M_{0,0}(X,e)$ is irreducible and has the  expected dimension $\mu_1(e)=e(n-d)+n-5$,  provided that $n\geq d+3$.
Recent work of Bilu  and Browning \cite[Cor.~1.4]{bilu}
achieves the same conclusion for any  smooth hypersurface $X\subset \PP^{n-1}$ of degree $d$, provided that
$n> (d-1)2^{d}$, improving on work of  Browning and Sawin \cite[Thm.~1.1]{BS2} who 
required $n> (2d - 1)2^{d-1}$. The latter result is proved using analytic number theory over function fields and, in turn,  builds on an approach
employed by Browning and Vishe \cite{BV'}. The idea is to study
the moduli space of morphisms
$\Mor_e(\PP^1,X)$, whose expected dimension is $\mu_1(e)+3$, since
$\mathcal{M}_{0,0}(X,e)$ is obtained from $\Mor_e(\PP^1,X)$ 
on taking the quotient by $\mathrm{PGL}_2$. 
To compute the dimension of $\Mor_e(\PP^1,X)$, it suffices to work over a finite field $\FF_q$ of  characteristic $>d$, estimate the cardinality
$\#\Mor_e(\PP^1,X)(\FF_q)$ using analytic number theory, and then compare it with the statement of the Lang--Weil estimate. The  application of function field analytic number theory to the geometry 
of the moduli spaces of degree $e$ morphisms to $X$
has  been 
implemented in a range of contexts, provided that $n$ is sufficiently large in terms of $d$ and $e$. These applications include:
 \begin{itemize}
 \item
The dimension of the  singular locus of
 $\mathcal M_{0,0}(X,e)$ (Browning--Sawin \cite{BS2});
 \item
 The geometry of  $\mathcal M_{0,m}(X,e)$ when $X$ is a complete intersection cut out by more than one equation  (Browning--Vishe--Yamagishi \cite{BVY});
\item
The geometry of the moduli space
 $\mathcal M_{g,0}(X,e)$
of degree $e$ morphisms from a smooth
projective curve of genus $g$ to $X$ (Hase-Liu \cite{hase1});
\item
The singularities of $\mathcal M_{g,0}(X,e)$ (Glas--Hase-Liu \cite{hase2}).
 \end{itemize}

While the geometry of rational curves on hypersurfaces is now well understood in many regimes, the analogous study of rational surfaces is  largely undeveloped.
As described by Lang \cite{lang},
a result of Tsen shows that $X$ always contains a rational surface
if $n> d^2$.
By contrast,
when $X$ is very general and $n> \max\{20,d\}$, for a suitable range of
 $e$ compared to  $d$ and $n$, it follows from work of
Beheshti  and Riedl \cite{roya'} that there are no
rational surfaces in $X$ that are ruled by low-degree rational curves.

Let $\Mor_e(\PP^2,X)$
denote the  moduli space
 of degree $e$ morphisms $g:\PP^2\to X$.
Such a morphism is given by
$
g=(g_1,\dots,g_n),
$
where $g_1,\dots,g_n\in K[u,v,w]$ are  ternary forms of degree $e$,
with no common zero in $\PP^2$, such that $f(g_1,\dots,g_n)$ vanishes  identically.
Since the number of monomials of degree $D$ in three variables is $\binom{D + 2}{2}$,
we may regard $g$ as a point in $\PP^{n \binom{e + 2}{2}-1}$ and
the morphisms of degree $e$ on $X$ are parameterised by
$\Mor_e(\PP^2,X)$, which is an open subvariety of
$\PP^{n \binom{e + 2}{2}-1}$ cut out by a system of $\binom{d e + 2}{2}$ equations of degree $d$.
This shows that 
either  $\Mor_e(\PP^2,X)$ is empty or else it has dimension at least $\mu(e)$,
where
\begin{equation}\label{eq:mu}
\mu(e) = n \binom{e + 2}{2}  -  \binom{d e + 2}{2} -1
\end{equation}
is the {\em expected dimension} of $\Mor_e(\PP^2,X)$. 

Although it does not address
the geometry of $\Mor_e(\PP^2,X)$, the most relevant work in this direction is due to Starr \cite{starr},
who focuses on the moduli space
$F_2(e,X)$ of degree $e$ {\em Veronese surfaces} contained in $X$. These are rational surfaces
embedded in $X$ via morphisms induced by complete linear systems of degree $e$ homogeneous forms.
(In fact, Starr's work treats the more general case of  degree $e$ Veronese $r$-folds, for any $r\in \NN$, but in this discussion we restrict our attention to the case $r=2$.)
Note that when $e = 1$, we recover the classical Fano variety of planes $F_2(X) = F_2(1, X)$  in $X$, which  is obtained from  $\Mor_1(\PP^2,X)$ on taking the quotient by $\mathrm{PGL}_{3}$.
It follows from \cite[Thm.~1.3]{starr} that $F_2(e,X)$ has the expected dimension for a sufficiently general hypersurface
$X\subset \PP^{n-1}$ of degree $d$, provided that
\begin{equation}
\label{starr1}
n\geq \frac{(de+2)(de+1)}{(e+2)(e+1)}+(e+2)(e+1).
\end{equation}
Similarly, the irreducibility of $F_2(e,X)$ can be deduced from \cite[Cor.~1.5]{starr} and \cite[Thm.~1.6(i)]{starr}
under the more stringent condition that
\begin{equation}\label{starr2}
n\geq n_e+\frac{1}{n_e+1}\binom{n_e+d}{n_e},
\end{equation}
where $n_e=\frac{1}{2}e(e+3)$.
For fixed $d$,  Stirling's formula shows that the right hand side is asymptotic to $e^{2(d-1)}/(\exp(d)2^{d-1}d!)+\mathbf{1}_{d=2}e^2/2$, as $e\to \infty$.

More is known about the  Fano variety of planes $F_2(X)$. According to work of
Hochster and Laksov \cite{hoch},
for a general hypersurface
$X$, the variety $F_2(X)$ is irreducible and has the expected dimension $\mu(1)-8=3n-\binom{d+2}{2}-9$, whenever this quantity is positive. In 
 fact, it follows from  \cite[Thm.~6.28]{3264} that $F_2(X)$ is non-empty if $d \geq 3$ and $3(n-3) - \binom{d + 2}{2} \geq 0$.

\begin{remark}
We can give simple conditions under which 
$\Mor_e(\PP^2,X)$ is non-empty, for any hypersurface $X\subset \PP^{n-1}$ of degree $d$.
Suppose first that 
$d \geq 3$ and 
$n\geq \frac{1}{3}\binom{d+2}{2} + 3$.
Then $F_2(X)\neq \emptyset$ by \cite[Thm.~6.28]{3264}, whence  $\Mor_1(\PP^2,X)\neq \emptyset$. But then, on composing any morphism in $\Mor_1(\PP^2,X)$ with one from $\Mor_e(\PP^2, \PP^2)$, we obtain a morphism in $\Mor_e(\PP^2,X)$. 
If  $d=2$ and $n\geq 6$
then 
$F_2(X)\neq \emptyset$ by \cite[Thm.~22.13]{harris}, whence $\Mor_e(\PP^2, X)$ is non-empty if $n\geq 6$.
\end{remark}

We are now ready to  state our main result, which  for the first time 
establishes the irreducibility and dimension of $\Mor_e(\PP^2,X)$ in suitable regimes,
for arbitrary smooth hypersurfaces $X\subset \PP^{n-1}$.
  In the case $e=1$ we also record an 
easy  consequence for the Fano variety of planes.

\begin{theorem}\label{t:main}
Let $d \geq 2$ and $e \geq 1$ be integers and let  $K$ be an algebraically closed field whose characteristic is $0$ or exceeds $d$.
Let
$X\subset \PP^{n-1}$ be a smooth hypersurface of degree $d$ defined over  $K$ such that
$n > 2^{d}(d-1)(de+1).$
Then $\Mor_e(\PP^2,X)$
is  irreducible and has dimension $\mu(e)$.
\end{theorem}

\begin{corollary}
Let $d \geq 2$ and let  $K$ be an algebraically closed field whose characteristic is $0$ or exceeds $d$.
Let $X\subset \PP^{n-1}$ be a smooth hypersurface of degree $d$ defined over  $K$ such that $n > 2^{d}(d^2-1).$
Then $F_2(X)$ is  irreducible and has dimension $3n-\binom{d+2}{2}-9$.
\end{corollary}

It is worth emphasising that the dependence of $n$ on $e$ is linear 
in Theorem \ref{t:main}, compared to the polynomial dependence on $e$ that occurs in 
Starr's result for the moduli space of  degree $e$ Veronese surfaces in a sufficiently  general hypersurface $X$,
with the constraints in \eqref{starr1} and \eqref{starr2}.
In Proposition~\ref{p:small} we shall prove
that
$\dim \Mor_e(\PP^2,X)>\mu(e)$ when $K=\bar\FF_p$ and $p$ is sufficiently small compared to $d$ and $e$, so that some lower bound on the characteristic is necessary when it is positive.

\medskip

Let us now describe the main ideas of the proof. By spreading out,
we may reduce the problem to estimating the quantity
\begin{equation}
\label{big*}
N(e) = \#\left\{
 \underline{ \mathbf{t} } = (\t_0, \dots, \t_e) \in \FF_q[u]^{n (e + 1) } :
\begin{array}{l}
|\t_s|  < q^{s + 1} \text{ for all $0 \leq s \leq e$}  \\
F_j(  \underline{ \t }  ) = 0 \text{ for all $0 \leq j \leq d e$}
\end{array}
\right\},
\end{equation}
where $F_j(  \underline{ \t }  )$ are the degree $d$ forms with coefficients in $\FF_q$ defined in \eqref{eq:ash}. 
For fixed integers $N,R\geq 1$, 
consider the general problem of 
estimating the quantity
$$
\#\left\{
 \x \in \FF_q[u]^{ N } :
\begin{array}{l}
|x_i|  < q^{e + \delta_i} \text{ for all $1 \leq i \leq N$}  \\
G_j(  \x  ) = 0 \text{ for all $1 \leq j \leq R$}
\end{array}
\right\},
$$
where $G_j(\x)$ are degree $d$ forms  with coefficients in $\FF_q$,  and  $\delta_1,\dots,\delta_N$
are real numbers.
This is exactly the setting of the function field analogue \cite{lee} 
of a classical result by  Birch \cite{B}. The latter is capable of providing an asymptotic formula for this counting function (as $e\to \infty$) when 
$|\delta_1|,\dots,|\delta_N|$ are sufficiently small with respect to $e$,
provided that $N-B\geq 2^{d-1}(d-1)R(R+1)$, where $B$ is the dimension of the {\em Birch singular locus}
\begin{equation}
\label{sing}
\left\{ \x \in \AA^N: \textnormal{rank} (  \textnormal{Jac}_\mathbf{G} (\x)  ) < R  \right\}.
\end{equation}
In fact, it follows from  \cite{BVY} that an asymptotic formula is also possible when $N-B$ grows only linearly in $R$.
In our setting, we have $N=n(e+1)$ and $R=de+1$, both of which depend on $e$. 
Two major obstacles  obstruct a direct  application of \cite{lee} or \cite{BVY}. First,  the side lengths of the boxes vary wildly in our setting, ranging from $q$ to $q^{e+1}$. Second, we have $B\geq en$
in our setting, 
since it turns out that
the variety  (\ref{sing}) contains all vectors
$\underline{ \mathbf{t} } \in \AA^{n (e + 1) }$ with  $\t_0 = \0$.

Let $\TT$ be the function field analogue of the unit interval. Then, 
as described in Section~\ref{sec:proof}, 
the starting point of the circle method 
is the identity
$$
N(e)=\int_{\TT^{de+1}} S(\bal) \d \bal,
$$
where $\bal=(\alpha_0,\dots,\alpha_{de})$ and 
$
S(\bal)$ is the exponential sum defined in  (\ref{SSS}).
For any choice of   $j\in \{0,\dots,de\}$ we can isolate   the portion of the phase function involving $\alpha_j$ via an expression of the form
$$
|S(\bal)|^{2^{d-1}} = \prod_{j = 0 }^{de} |S(\bal)|^{ \frac{2^{d-1}}{de + 1} }  \leq \prod_{j = 0 }^{de} |T_j(\alpha_j)|^{ \frac{1}{de + 1} },
$$
where  $T_j(\alpha_j)$ is an exponential sum whose phase function depends only on $\alpha_j$, and is   independent of $\alpha_0, \dots, \alpha_{j-1},$ $\alpha_{j+1}, \dots, \alpha_{de}$.
This part of the argument exploits the specific bihomogeneous structure of the forms $F_j$  and in the case $e=1$ is reminiscent of arguments used by Brandes \cite{JuBr} to study  linear spaces on hypersurfaces. 
This decoupling allows us to factor the integral as
$$
\int_{\TT^{de+1}}|S(\bal)| \d \bal \leq
\prod_{j = 0 }^{de}  \int_{\TT} |T_j(\alpha_j)|^{ \frac{1}{ (de + 1) 2^{d-1}} } \d \alpha_j,
$$
which  reduces the problem to the hypersurface case.

One advantage of this approach is that 
the dimension of the Birch singular locus (\ref{sing}) of the resulting system is essentially optimal, as we shall discuss further in Remark \ref{sec:finalrem}. 
A second advantage is that it  eliminates the issue of the {\em lopsidedness} of the box.
The latter causes serious technical issues, and  even a difference of $q$ in side lengths
introduces inefficiencies into the argument that result in excess factors of $q$.  
In our approach,  Weyl differencing is applied in such a way that the problem is essentially reduced to
studying an exponential sum with the phase function given by  a bihomogeneous form $G(\x; \y)$, where the summation is over $|\x| < q^{P_1}$ and $|\y| < q^{P_2}$.
This topic has been addressed by Schindler  \cite{DS} over number fields, and we  present 
a function field version that significantly extends her ideas in Section \ref{sec:BHMG}.

We believe that the main ideas of this paper, or a suitable generalisation, have the potential to be useful in studying 
$\Mor_e(\PP^r, X)$ for $r > 2$, as well as in obtaining an upper bound of the correct order of magnitude for the number of integral points on suitable complete intersections where the circle method is not directly  applicable.

\begin{summary}
In Section \ref{s:prem} we shall collect together preliminary facts about function fields, together 
with a uniform estimate for point-counting over function fields in Lemma \ref{dimbdd} and 
a function field version of a basic Weyl differencing argument due to Schmidt in Lemma~\ref{Diff}. 
 Section \ref{s:key} is concerned with the reduction of Theorem \ref{t:main} to the problem of 
 estimating $N(e)$ in \eqref{big*}
and showing that the statement of  Theorem \ref{t:main} is false in very small positive characteristic.
   Section~\ref{sec:BHMG} gives a very general treatment of Weyl differencing for  
polynomials that satisfy a certain symmetry recorded in Hypothesis~\ref{HYP}. Finally, in Section \ref{sec:proof} we complete the proof of Theorem \ref{t:main} by establishing the asymptotic formula for $N(e)$ as $q\to \infty$.
\end{summary}

\begin{ack}
The authors are very grateful to Jakob Glas, Matthew Hase-Liu,  Eric Riedl and Will Sawin for useful comments.
While working on this paper the  first author was supported  by
a FWF grant (DOI 10.55776/P36278).
\end{ack}

\section{Preliminary facts about function fields}
\label{s:prem}
In this section,  we collect together some basic notation and facts concerning the function field $\KK = \FF_q(u)$, where $\FF_q$ is a finite field of characteristic $p$.

\subsection*{Notation}
We shall need  absolute value
$
|a/b|=
q^{\deg a-\deg b},
$
for any $a/b\in \KK^*$. We extend these definitions to  $\KK$ by taking $|0|=0.$
Let  $\KK_\infty$ be  the completion of $\KK$  with respect to $|\cdot|$.
We can extend the absolute value to get  a non-archimedean
absolute value
$|\cdot|: \KK_\infty\rightarrow \RR_{\geq 0}$
given by $|\alpha|=q^{\ord \alpha}$, where $\ord \alpha$ is the largest $i\in
\ZZ$ such that $a_i\neq 0$ in the
representation $\alpha=\sum_{i\leq M}a_i u^i$.
In this context we adopt the convention $\ord 0=-\infty$ and $|0|=0$.
We extend this to vectors by setting
$
|\x|=\max_{1\leq i\leq n}|x_i|,
$
for any $\x\in \KK_\infty^n$, noting that it satisfies  the ultrametric inequality
$
|\x + \y| \leq \max \{ |\x|, |\y| \},
$
for any $\x, \y \in \KK_\infty^n$.

We may identify $\KK_\infty$ with the set
$$
\FF_q((u^{-1}))=\left\{\sum_{i\leq M}a_i u^i: \text{for $a_i\in \FF_q$ and some $M \in\ZZ$} \right\}
$$
and put
$$
\TT=\{\alpha\in \KK_\infty: |\alpha|<1\}=\left\{\sum_{i\leq -1}a_i u^i: \text{for $a_i\in \FF_q$}
\right\}.
$$
Since $\TT$ is a locally compact
additive subgroup of $\KK_\infty$ it possesses a unique Haar measure $\d
\alpha$, which is normalised so that
$
\int_\TT \d \alpha=1.
$
We can extend $\d\alpha$ to a (unique) translation-invariant measure on $\KK_\infty$ in
such a way that
$$
\int_{\{\alpha\in \KK_\infty :
|\alpha|< q^M
\}} \d \alpha= q^M,
$$
for any $M \in \ZZ$.
These measures also extend to $\TT^n$ and $\KK_\infty^n$, for any $n\in \NN$.
Given $\alpha \in \KK_\infty$ we define $\| \alpha  \| = |\{ \alpha \}|$, where $\{\alpha\} \in \TT$ is the {\em fractional part} of $\alpha$.
Given any finite set $S$ we shall interchangeably write $\#S$ or $|S|$ to denote its cardinality.

\subsection*{Characters}
There is a non-trivial additive character $e_q:\FF_q\rightarrow \CC^*$ defined
for each $a\in \FF_q$ by taking
$e_q(a)=\exp(2\pi i \tr(a)/p)$, where $\tr: \FF_q\rightarrow \FF_p$ denotes the
trace map.
This character induces a non-trivial (unitary) additive character $\psi:
\KK_\infty\rightarrow \CC^*$ by defining $\psi(\alpha)=e_q(a_{-1})$ for any
$\alpha=\sum_{i\leq M}a_i u^i$ in $\KK_\infty$.
We have the  basic orthogonality property
\begin{equation}\label{eq:laurel}
\sum_{\substack{b\in \cO\\ |b|< q^N}}\psi(\gamma b)=\begin{cases}
q^N & \text{if $|\gamma|< q^{-N}$,}\\
0 & \text{otherwise},
\end{cases}
\end{equation}
for any $\gamma \in \TT$ and any integer $N\geq 0$, as proved in  \cite[Lemma 7]{kubota}.
We also have
\begin{equation}\label{eq:ortho}
\int_{\{\alpha\in \KK_\infty: |\alpha|<q^M\}} \psi(\alpha \gamma) \d \alpha =\begin{cases}
q^M &\text{ if $|\gamma|<q^{-M}$,}\\
0 &\text{ otherwise,}
\end{cases}
\end{equation}
for any $\gamma\in \KK_\infty$ and $M \in \ZZ$, as proved in \cite[Lemma 1(f)]{kubota}.

\subsection*{Counting in function fields}
Given a variety $X \subset \AA^n$ defined over $\overline{\FF_q(u)}$, we define
$$
\delta(X) = \deg X_1 + \cdots + \deg X_s,
$$
where $X = \bigcup_{1 \leq  i \leq s} X_i$ is the decomposition of $X$ into its geometrically irreducible components. Note that $\delta(X) = \deg X$ if $X$ is equidimensional and the B\'{e}zout inequality $\delta(X \cap Y) \leq \delta(X) \delta (Y)$ always holds. We write
$$
X(J) =  X  \cap  \{ \x \in \FF_q[u]^n: |\x|< q^J   \},
$$
for any  $J \in \mathbb{N}$.
The following result extends to $\FF_q(u)$ a familiar finite field counting result that was recorded  in \cite[Lemma~2.1]{BVY}.

\begin{lemma}
\label{dimbdd}
Let $X \subset \mathbb{A}^n$ be a variety defined over  $\overline{\FF_q(u)}$. Then
$$\# X(J)  \leq \delta(X) q^{J \dim X},
$$
for any $J \geq 1$.
\end{lemma}
\begin{proof}
We argue by induction on $r = \dim X$. When $r=0$ then $X$ consists of at most $\deg X = \delta(X)$ points and the lemma is trivial.
Assume that $r \geq 1$ and let
$Z$ be an irreducible component of $X$, with  $\dim Z \geq 1$.
We shall show that there is an index $1 \leq i \leq n$ such that
the intersection of $Z$ with  the hyperplane $H_{y} = \{ \x \in \AA^n:  x_i = y \}$ satisfies
$$
\dim (Z \cap H_{y} ) < \dim Z,
$$
for any $y \in \FF_q[u]$ such that $|y|<q^J$. Suppose that the opposite is true, so that there exist $y_1,\dots, y_n \in \FF_q[u]$  such that
$|y_1|,\dots,|y_n|<q^J$ and
$
Z \subset \{ \x \in \AA^n: x_j = y_j \},
$
for each $1 \leq j \leq n$.
However, this means that $Z = \{ (y_1, \ldots, y_n) \}$, which contradicts the assumption that $\dim Z \geq 1$.
Since $Z \cap H_{y}$ has dimension at most $\dim Z -1$ and  $\delta ( Z \cap H_{y} )  \leq \deg Z$, the induction hypothesis yields
$$
\# Z(J) \leq \sum_{|y| < q^J } \#  (Z \cap H_{y} ) (J) \leq \sum_{|y| < q^J} (\deg Z) q^{J (\dim Z - 1)}  \leq  (\deg Z) q^{ J \dim Z}.
$$
Therefore, if we denote by $Z_1, \ldots, Z_s$ the irreducible components of $X$, we obtain
$$
\# X(J) \leq
\sum_{ i  = 1}^s \# Z_i(J) \leq \sum_{ i = 1}^s (\deg Z_i) q^{J \dim Z_i}  \leq  \sum_{i = 1}^s (\deg Z_i) q^{J r} .
$$
The result follows since $\delta (X) = \sum_{ i  = 1}^s \deg Z_i$.
\end{proof}

\subsection*{Weyl differencing}
In this section, we recollect results from \cite[Sect.~11]{S} in the $\FF_q[u]$-setting. Since these results
over $\FF_q[u]$ can be deduced in essentially the same manner, provided $p > d$, we omit the details.
Given a box $\mathcal{Z} \subset \mathbb{K}_\infty^{N}$, Schmidt introduces the 
sets  $\mathcal{Z}^D = \mathcal{Z} - \mathcal{Z} = \{ \mathbf{z} - \mathbf{z}' : \mathbf{z}, \mathbf{z}' \in \mathcal{Z} \}$ and 
$$
\mathcal{Z} (\mathbf{z}_1, \ldots, \mathbf{z}_t) = \bigcap_{\epsilon_1 \in \{0, 1\} } \cdots \bigcap_{\epsilon_t \in \{0, 1\} }
(\mathcal{Z} - \epsilon_1 \mathbf{z}_1 - \cdots - \epsilon_t \mathbf{z}_t),
$$
for any given  $\mathbf{z}_1, \ldots, \mathbf{z}_t \in \mathbb{K}_\infty^{N}$. 
In our function field setting, the ultrametric inequality implies that in fact 
$
\mathcal{Z}^D = \mathcal{Z}$ and
$\mathcal{Z}(\mathbf{z}_1, \ldots, \mathbf{z}_{t})  = \mathcal{Z}$, provided that 
$\mathbf{z}_1, \ldots, \mathbf{z}_{t}\in \mathcal{Z}$.
We may now record the  following lemma, which is the $\FF_q[u]$-analogue of \cite[Lemma~11.1]{S}.

\begin{lemma}
\label{Diff}
Let $\mathcal{P} \in \KK_{\infty} [z_1, \dots, z_N]$ be a non-constant polynomial,
let $\mathcal{Z}\subset \KK_\infty^N$ be a box and let 
$$
S = \sum_{ \z \in \mathcal{Z} } \psi(  \mathcal{P} (\z)  ).
$$
For any  $t \in \NN$, we have 
$$
|S|^{2^{t-1}} \leq   |\mathcal{Z}|^{2^{t-1} - t} \sum_{ \z_1 \in \mathcal{Z}} \dots  \sum_{ \z_{t-1} \in \mathcal{Z}}
\left|  \sum_{ \z_{t} \in \mathcal{Z}}   \psi(  \mathcal{P}_t(\z_1, \ldots, \z_t ) )    \right|,
$$
where 
$$
\mathcal{P}_{t}(\z_1, \ldots, \z_{t}) = \sum_{\epsilon_1 \in \{0, 1\} } \cdots \sum_{\epsilon_{t} \in \{0, 1\} } \,
(-1)^{\epsilon_1 + \cdots + \epsilon_{t} }
\mathcal{P}( \epsilon_1 \z_1 + \cdots + \epsilon_{t} \z_{t} ).
$$
\end{lemma}

It follows from \cite[Lemma 11.2(a)]{S} that if $t > \deg \mathcal{P}$, then
\begin{equation}
\label{ZERO}
\mathcal{P}_t( \z_1, \ldots, \z_{t} ) \equiv 0;
\end{equation}
i.e.\ it is the zero polynomial. Furthermore, it follows from \cite[Lemma 11.4(A)]{S} that if $t = \deg \mathcal{P}$,
then $\mathcal{P}_{t} (\z_1, \ldots, \z_{t})$ is the unique symmetric multilinear form associated to $\mathcal{P}^{[t]}(\z)$, the degree $t$ component of $\mathcal{P}(\z)$; i.e.\
$\mathcal{P}_{t} (\z_1, \ldots, \z_{t})$  satisfies
\begin{equation}
\label{symmm}
\mathcal{P}_{t}(\mathbf{z}, \ldots, \mathbf{z}) = (-1)^{t}   t!  \mathcal{P}^{[t]} (\mathbf{z}).
\end{equation}

\section{Reduction to a counting problem }\label{s:key}

We revisit the argument of \cite{BV'} to establish the irreducibility
 and dimension of the moduli space $\Mor_{e}(\PP^2,X)$.
For  $d\geq 2$ and $e\geq 1$, let
 $K$ be a field whose characteristic is $0$ or exceeds $d$.
Let
$X\subset \PP^{n-1}$ be a smooth hypersurface of degree $d$  defined
by a form
 $f \in K[x_1,\dots,x_n]$.
The argument  in Section \ref{s:intro} shows that
$$
\dim \Mor_e(\PP^2,X)\geq \mu(e),
$$ 
where $\mu(e)$ is given by
\eqref{eq:mu}.
Thus it suffices to show that
 $\Mor_e(\PP^2,X)$
 is irreducible with  $\dim \Mor_e(\PP^2,X)\leq \mu(e)$.
 Spreading out to a finite field $\FF_q$, with characteristic $p>d$, it suffices to do this
 when everything is defined over  $\FF_q$.

We henceforth fix a finite field $\FF_q$ of
characteristic $p>d$ and we suppose that $f\in \FF_q[x_1,\dots,x_n]$ is a non-singular degree $d$ form defined over it.
Let us write
\begin{equation}\label{eq:cherry}
f(\x) =  \sum_{i_1,\dots,i_d=1}^n c_{i_1,\dots,i_d}  x_{i_1}  \dots  x_{i_d},
\end{equation}
with coefficients $c_{i_1,\dots,i_d} \in \FF_q$ which are symmetric in the indices.
We also denote
\begin{equation}
\label{unique}
\Gamma_f (\x_1, \dots, \x_d) =  
\sum_{i_1,\dots,i_d=1}^n c_{i_1,\dots,i_d}  x_{1, i_1}  \dots  x_{d, i_d}.
\end{equation}
The cone over  $\Mor_e(\PP^2,X)$ is
the space of tuples of (not necessarily homogeneous)
 polynomials $g_1,\dots,g_n \in \FF_q[u,v]$
of degree at most $e$, at least one of degree exactly $e$, with no common zero
shared by $g_1,\dots,g_n$ in  $\overline\FF_q$, nor by the  homogeneous degree $e$ parts
of $g_1,\dots,g_n$, and
such that $f(g_1,\dots,g_n)=0$.
Any degree $e$ polynomial $g_i\in \FF_q[u,v]$ takes the shape
$
g_i (u, v) = g_{i,0} v^e + g_{i,1}  v^{e - 1} + \cdots + g_{i,e},
$
where $g_{i,s} \in \FF_q[u]$ with $\deg g_{i,s} \leq s$, for all $1 \leq i \leq n$ and $0 \leq s \leq e$.
Then
\begin{align*}
f(\g)
&=
\sum_{i_1,\dots,i_d=1}^n c_{i_1,\dots,i_d}
g_{i_1}(u, v) \dots g_{i_d}(u,v)
\\
&=
\sum_{i_1,\dots,i_d=1}^n c_{i_1,\dots,i_d}
\left( \sum_{s_1 = 0}^e g_{i_1, s_1}  v^{e - s_1} \right) \dots \left( \sum_{s_d = 0}^e g_{i_d, s_d}  v^{e - s_d} \right)
\\
&=
\sum_{j = 0}^{d e}  v^{de - j}  \sum_{i_1,\dots,i_d=1}^n c_{i_1,\dots,i_d}
\sum_{ \substack{ 0 \leq s_1, \dots, s_d \leq e \\s_1 + \dots + s_d = j }}  g_{i_1, s_1}   \dots  g_{i_d, s_d}.
\end{align*}
But then it follows that 
$$
f(\g)=
\sum_{j = 0}^{d e}  v^{de - j}
F_j(\underline{ \g }),
$$
where
\begin{equation}\label{eq:ash}
F_j( \underline{ \g }  ) =
\sum_{ \substack{ 0 \leq s_1, \dots, s_d \leq  e \\ s_1 + \dots + s_d = j }} \sum_{i_1,\dots,i_d=1}^n c_{i_1,\dots,i_d}   g_{i_1, s_1}   \dots  g_{i_d, s_d} ,
\end{equation}
for $0 \leq j \leq d e$, and where
$\underline{ \g } = (\g_0, \dots, \g_{e})$
with  $\g_s = (g_{1,s}, \dots, g_{n, s})$ for all $0 \leq s \leq e$.
Note that when $e=1$, the same forms appear in  work of Brandes \cite{JuBr} over $\QQ$.

Let $M(q,e)$ be
the space of tuples of (not necessarily homogeneous)
 polynomials $g_1,\dots,g_n \in \FF_q[u,v]$
of degree at most $e$, at least one of the degrees exactly $e$, with no common zero
shared by $g_1,\dots,g_n$ in    $\overline\FF_q$, nor by the  homogeneous degree $e$ parts
of $g_1,\dots,g_n$, and
such that $f(g_1,\dots,g_n)=0$.
Then we have
$$
\#\Mor_e(\PP^2,X)(\FF_q)=
\frac{\#M(q,e)}{q-1}
$$
and
the expected dimension of $M(q,e)$ is $\hat\mu(e)$, where
\begin{equation}\label{eq:mu'}
\hat\mu(e) = n \binom{e + 2}{2}  -  \binom{d e + 2}{2}.
\end{equation}
Appealing to the Lang--Weil estimate, as in the proof of
  \cite[Eq.~(3.3)]{BV'},
 in order to deduce that
 $\Mor_e(\PP^2,X)$
 is irreducible and has the expected dimension it will  be enough to show that
$$
 \lim_{\ell\to \infty}  q^{-\ell\hat\mu(e)} \#M(q^\ell, e)\leq 1.
$$
To this end, on dropping some of the conditions in the definition of
$M(q,e)$, it will suffice to
establish the existence of
$\delta>0$ such that
\begin{equation}\label{eq:goat}
N(e) \leq
q^{\hat \mu(e)}+O(q^{\hat \mu(e)-\delta}),
\end{equation}
where
the implied constant is independent of $q$ and
\begin{align*}
N(e) &= \#\left\{
\mathbf{g}\in \FF_q[u, v]^n :
\begin{array}{l}
\deg g_1,\dots,\deg g_n \leq e\\
f(\g) = 0\\
 \end{array}
 \right\}
 \\
&= \#\left\{
 \underline{ \mathbf{g} } \in \FF_q[u]^{n (e + 1) } :
\begin{array}{l}
|\g_s|  < q^{s + 1} \text{ for all $0 \leq s \leq e$}  \\
F_j(  \underline{ \g }  ) = 0 \text{ for all $0 \leq j \leq d e$}
\end{array}
\right\},
\end{align*}
which is exactly the quantity that was defined in (\ref{big*}).
We shall prove the following result in Section \ref{sec:proof}.

\begin{theorem}\label{t:2}
Let $d\geq 2$ and $e\geq 1$ be integers. 
Let  $f \in \FF_q[x_1, \dots, x_n]$ be a non-singular degree $d$ form defined over a finite field  $\FF_q$ of characteristic $p> d$. Suppose that $n > 2^{d} (d-1)(de+1)$.
Then there exists $\delta > 0$ such that \eqref{eq:goat} holds,
where the implied constant depends on $d,n$ and $e$, but  is independent of $q$.
\end{theorem}

We conclude this section by showing that the statement of Theorem \ref{t:main} does not hold when $K$ is a field of small characteristic.
The following result draws inspiration from \cite[Lemma 2.1]{BS2}.

\begin{proposition}\label{p:small}
Let $K = \overline{\FF}_p$ for a prime $p$ and let $X \subset \PP^{n-1}$ be the Fermat hypersurface
$
x_1^d + \dots + x_n^d = 0.
$
Assume that $e \in \NN$ and $d = (e+1)! p + 1$. Then $X$ is smooth and
$
\dim \Mor_e (\PP^2, X) > \mu(e).
$
\end{proposition}

\begin{proof}
Firstly, the  Fermat hypersurface
is smooth over $K$ if and only if $p\nmid d$. For the second part, we take   $f (\x) = x_1^d + \dots + x_n^d$ and note that 
\eqref{eq:ash}  becomes
$$
F_j( \underline{ \g }  ) =
\sum_{ \substack{ 0 \leq s_1, \dots, s_d \leq  e \\ s_1 + \dots + s_d = j }} \sum_{i=1}^n    g_{i, s_1}   \dots  g_{i, s_d},
$$
for  $0 \leq j \leq d e$, where we recall that
$g_{i,s} \in \FF_q[u]$ satisfies  $\deg g_{i,s} \leq s$, for all $1 \leq i \leq n$ and $0 \leq s \leq e$.
Our assumption on $d$ implies that $d>e+1$.
For $j = e + 1$ we obtain
\begin{align*}
F_{e+1} ( \underline{ \g }  ) =~&
\sum_{ \substack{ 0 \leq s_1, \dots, s_d \leq  e \\ s_1 + \dots + s_d = e + 1 }} \sum_{i=1}^n    g_{i, s_1}   \dots  g_{i, s_d}
\\
=~&
\sum_{  e r_e + \cdots +  2r_2+ r_1 = e + 1   }
 \binom{d}{r_e}   \binom{d - r_e}{r_{e-1}}  \dots  \binom{d - (r_e + \dots + r_2)}{r_1 }\\
&\qquad \times  \sum_{i=1}^n    g_{i, e}^{r_e}   \dots   g_{i,1}^{r_1} g_{i,0}^{d -
(e + 1) }.
\end{align*}
Suppose $r_e=\cdots=r_2=0$
 in the summation. Then $r_1 = e + 1$ and the summand becomes
$$
\binom{d}{e + 1 }
\sum_{i=1}^n g_{i,1}^{e + 1} g_{i,0}^{d -
(e + 1) }.
$$
Since $e + 1 \geq 2$, we see that  if $d = (e+1)! p + 1$ then the binomial coefficient $\binom{d}{e + 1 }$ is divisible by $p$.
On the other hand, if $r_e + \dots + r_2 \geq 1$, then one of the binomial coefficients
$$
\binom{d - (r_e + \dots + r_j)}{r_{j-1} } = \frac{  (d - (r_e + \dots + r_j) ) \cdots ( d - (r_e + \dots + r_j) - r_{j-1} + 1 )  }{r_{j-1} !}
$$
must contain $d-1$ in the numerator,
for some $j\in \{2,\dots, e+1\}$.
Thus, since $r_{j-1} \leq e + 1$, this binomial coefficient is divisible by $p$.
We have therefore shown that $F_{e+1}$ vanishes in $K$ and so 
the space $\Mor_e(\PP^2,X)$ is cut out by at most $\binom{de + 2}{2} - 1$ equations in
$n \binom{e + 2}{2}$ variables, whence $\dim \Mor_e(\PP^2,X) >  \mu(e).$
\end{proof}

\section{Weyl differencing and a mean value estimate}
\label{sec:BHMG}

Let $d \geq 2$ be an integer and let 
 $\FF_q$ be a finite field of  characteristic $p>d$.
Let $d_1,d_2\in \ZZ$ such that $d_1\geq 0$ and $d_2\geq 1$, with $d = d_1 + d_2$.
Let 
$\x=(x_1, \ldots, x_n)$ and $\y=( y_1, \ldots, y_n)$ and let  $F(\x, \y)$
be a degree $d$ polynomial of the form
$$
F(\x, \y) = G (\x; \y)  + \mathfrak{f}(\x, \y) + \mathfrak{g}(\x, \y),
$$
where $G (\x; \y) \in \FF_q [x_1, \ldots, x_n, y_1, \ldots, y_n]$ is bihomogeneous of bidegree $(d_1, d_2)$ and
$\mathfrak{f}(\x, \y), \mathfrak{g}(\x, \y) \in \KK_\infty [x_1, \ldots, x_n, y_1, \ldots, y_n]$ are such that every term of $\mathfrak{f}(\x, \y)$  has degree in $\y$ strictly less than $d_2$ and every term of $\mathfrak{g}(\x, \y) $ has degree in $\x$ strictly less than $d_1$. (Note that this means that the degree in $\y$ can be strictly greater than $d_2$ in $\mathfrak{g}(\x,\y)$.)

For $\alpha \in \TT$ and integers $1 \leq P_1 \leq P_2$, we define
$$
T(\alpha) = \sum_{\x \in \mathcal{X}} \sum_{\y \in \mathcal{Y}}  \psi( \alpha F( \x, \y ) ),
$$
where
$$
\mathcal{X} = \{ \x \in \FF_q[u]^n:  |\x| < q^{P_1} \} \quad \textnormal{and} \quad  \mathcal{Y}  = \{ \y \in \FF_q[u]^n:  |\y| < q^{P_2} \}.
$$
We shall set $P_1 = P_2$ and $G(\x; \y) = G(\y)$ if $d_1 = 0$.
Let  $\mathbf{j} = (j_1, \ldots, j_{d_1})$ and $\mathbf{k} = (k_1, \ldots, k_{d_2})$. Then we may write 
\begin{align*}
G(\mathbf{x};\mathbf{y}) &= \sum_{j_1 = 1}^{n} \cdots \sum_{j_{d_1} = 1}^{n}
\sum_{k_1 = 1}^{n} \cdots \sum_{k_{d_2} = 1}^{n}
G_{\mathbf{j}, \mathbf{k}}  x_{j_1} \cdots x_{j_{d_1}} y_{k_1} \cdots y_{k_{d_2}}
\\
&= \sum_{1 \leq \mathbf{j} \leq n} \sum_{1 \leq \mathbf{k} \leq n}
G_{\mathbf{j}, \mathbf{k}}   x_{j_1} \cdots x_{ j_{d_1} }  y_{k_1} \cdots y_{ k_{d_2} },
\end{align*}
where  $G_{\mathbf{j}, \mathbf{k}} \in \FF_q$ is symmetric in $(j_1, \ldots, j_{d_1})$ and also in $(k_1, \ldots, k_{d_2})$.

\subsection{Weyl differencing} \label{sec:DIFF}
In this section, we make use of the results 
introduced in Section \ref{s:prem}, inspired by  work of Schindler \cite{DS}. However, we 
 remark that one additional step of Weyl differencing is required, 
 due to the fact that the degree $d$ component of the polynomial we consider is not necessarily bihomogeneous.
First, by H\"{o}lder's inequality we obtain
\begin{equation}
\label{ineq 1-1}
|T(\alpha)|^{ 2^{d_2} } \leq |\mathcal{X}|^{ 2^{d_2} - 1 } \sum_{ \substack{ \x  \in \mathcal{X} }}
| T_{\x} ({\alpha}) |^{2^{d_2 }},
\end{equation}
where
$$
T_{\x} (\alpha) = \sum_{ \y  \in \mathcal{Y} } \psi( \alpha F( \x, \y ) ).
$$
We proceed by using Lemma \ref{Diff} to bound $| T_{\mathbf{x}} (\alpha) |^{2^{d_2}}$.

Let
\begin{equation}
\label{defnF=g}
\mathcal{P}(\mathbf{y}) =  \alpha F( \x, \y )
\end{equation}
and
$$
\mathcal{F}(\y)
=
\alpha G(\x; \y)
= \alpha
 \sum_{1 \leq \mathbf{k} \leq n} \left( \sum_{1 \leq \mathbf{j} \leq n}
G_{\mathbf{j}, \mathbf{k}} x_{j_1} \cdots x_{j_{d_1}}   \right)  y_{k_1} \cdots y_{k_{d_2}}.
$$
For each $t \in \mathbb{N}$ we denote
$$
\mathcal{P}_{t}(\mathbf{y}_1, \ldots, \mathbf{y}_{t}) = \sum_{\epsilon_1 \in \{0, 1\} } \cdots \sum_{\epsilon_{t} \in \{0, 1\} } \,
(-1)^{\epsilon_1 + \cdots + \epsilon_{t} }
\mathcal{P}( \epsilon_1 \mathbf{y}_1 + \cdots + \epsilon_{t} \mathbf{y}_{t} ),
$$
and we set  $\mathcal{P}_0\equiv 0$.
In particular, it follows that
\begin{align*}
\mathcal{P}_{t}(\mathbf{y}_1, \ldots, \mathbf{y}_{t}) 
=~&
\mathcal{P}_{t-1}(\mathbf{y}_1, \ldots, \mathbf{y}_{t-1})\\
&-
\sum_{\epsilon_1 \in \{0, 1\} } \cdots \sum_{\epsilon_{t-1} \in \{0, 1\} } \,
(-1)^{\epsilon_1 + \cdots + \epsilon_{t-1} }
\mathcal{P}( \epsilon_1 \mathbf{y}_1 + \cdots + \epsilon_{t-1} \mathbf{y}_{t-1} + \mathbf{y}_{t} )
\\
=~&
\mathcal{P}_{t-1}(\mathbf{y}_1, \ldots, \mathbf{y}_{t-1})
-
\mathcal{P}_{t-1}(\mathbf{y}_1, \ldots,  \mathbf{y}_{t-2}, \mathbf{y}_{t-1} + \mathbf{y}_t )
\\
&
+ \mathcal{P}_{t-1}(\mathbf{y}_1, \ldots, \mathbf{y}_{t-2},  \mathbf{y}_{t}),
\end{align*}
for any $t \in \NN$.  But then, on taking $t = d_2 + 1$,  we may deduce
\begin{equation}\begin{split}
\label{relnnn}
\mathcal{P}_{d_2}(\mathbf{y}_1, \ldots, \mathbf{y}_{d_2-1}, \mathbf{y}_{d_2 + 1}) 
&- \mathcal{P}_{d_2}(\mathbf{y}_1, \ldots, \mathbf{y}_{d_2-1}, \mathbf{y}_{d_2} + \mathbf{y}_{d_2 + 1})
\\
&= \mathcal{P}_{d_2 + 1}(\mathbf{y}_1, \ldots, \mathbf{y}_{d_2}, \mathbf{y}_{d_2 + 1}) - \mathcal{P}_{d_2}(\mathbf{y}_1, \ldots, \mathbf{y}_{d_2}).
\end{split}
\end{equation}

Turning to the application of Lemma \ref{Diff}, 
it follows that  
\begin{equation} \label{beforeCS}
| T_{\x} ({\alpha}) |^{2^{d_2 - 1} }
\leq 
| \mathcal{Y} |^{2^{d_2-1} - d_2} \sum_{\y_1 \in \mathcal{Y}} \cdots
\sum_{\mathbf{y}_{d_2 - 1} \in \mathcal{Y}}
\left|  \sum_{\z \in \mathcal{Y} }
\psi ( \mathcal{P}_{d_2}(\mathbf{y}_1, \ldots, \mathbf{y}_{d_2 -1}, \z) ) \right|.
\end{equation}
By the Cauchy-Schwarz inequality, we deduce that
\begin{align*}
| T_{\mathbf{x}} ({\alpha}) |^{2^{d_2} }
&\leq  | \mathcal{Y} |^{2^{d_2} - d_2 - 1} \sum_{\mathbf{y}_1 \in \mathcal{Y}} \cdots
\sum_{\mathbf{y}_{d_2 - 1} \in \mathcal{Y}}
\left|  \sum_{\mathbf{z} \in \mathcal{Y}  }  \psi ( \mathcal{P}_{d_2}(\mathbf{y}_1, \ldots, \mathbf{y}_{d_2 - 1}, \mathbf{z}) ) \right|^2
\\
&=
| \mathcal{Y} |^{2^{d_2} - d_2 - 1} \sum_{\mathbf{y}_1 \in \mathcal{Y}} \cdots
\sum_{\mathbf{y}_{d_2 - 1} \in \mathcal{Y}}
\\
& \quad \times \sum_{\mathbf{z}, \mathbf{z}'   \in \mathcal{Y}  }
\psi ( \mathcal{P}_{d_2}(\mathbf{y}_1, \ldots, \mathbf{y}_{d_2 - 1}, \mathbf{z}) - \mathcal{P}_{d_2 }(\mathbf{y}_1, \ldots, \mathbf{y}_{d_2 - 1}, \mathbf{z}')  ).
\end{align*}
Writing $\z=\y_{d_1}$ and $\z'=\y_{d_2}+\y_{d_2+1}$, and 
recalling  \eqref{relnnn}, it follows that 
\begin{equation}\label{ineq 1-2}
| T_{\mathbf{x}} ({\alpha}) |^{2^{d_2} }\leq 
| \mathcal{Y} |^{2^{d_2 } - d_2 - 1 }
\sum_{\mathbf{y}_1 \in \mathcal{Y}}  \cdots
\sum_{\mathbf{y}_{d_2 + 1} \in \mathcal{Y} }  \psi( \mathcal{P}_{d_2 + 1}(\underline{\mathbf{y}}, \mathbf{y}_{d_2 + 1}) - \mathcal{P}_{d_2}(\underline{\mathbf{y}})
),
\end{equation}
where $\underline{\mathbf{y}} = (\mathbf{y}_1, \ldots, \mathbf{y}_{d_2})$.

Let
$$
\mathcal{F}_{d_2}(\underline{\mathbf{y}}) = \sum_{\epsilon_1 \in \{0, 1\} } \cdots \sum_{\epsilon_{d_2} \in \{0, 1\} }
(-1)^{\epsilon_1 + \cdots + \epsilon_{d_2} }
\mathcal{F}( \epsilon_1 \mathbf{y}_1 + \cdots + \epsilon_{d_2} \mathbf{y}_{d_2} )
$$
and
$$
\mathfrak{g}_{\x, t} (\mathbf{y}_1, \ldots, \mathbf{y}_{t}) = \sum_{\epsilon_1 \in \{0, 1\} } \cdots \sum_{\epsilon_{t} \in \{0, 1\} }
(-1)^{\epsilon_1 + \cdots + \epsilon_{t} }
\mathfrak{g} (\x,  \epsilon_1 \mathbf{y}_1 + \cdots + \epsilon_{t} \mathbf{y}_{t} ),
$$
for $d_2 \leq t \leq d_2 + 1$. We recall that the degree of $\mathfrak{f}(\x, \y)$ in $\y$ is strictly less than $d_2$.
Therefore, on making use of (\ref{ZERO}) and (\ref{symmm}), we obtain
$$
\mathcal{P}_{d_2}(\underline{\mathbf{y}}) = \mathcal{F}_{d_2}(\underline{\mathbf{y}}) +
\alpha \mathfrak{g}_{\x, d_2} (\underline{\mathbf{y}})
$$
and 
$$
\mathcal{P}_{d_2 + 1}(\underline{\mathbf{y}}, \mathbf{y}_{d_2 + 1}) = \alpha \mathfrak{g}_{\x, d_2 + 1} (\underline{\mathbf{y}}, \mathbf{y}_{d_2 + 1}).
$$
Moreover,  the polynomial $\mathcal{F}_{d_2} (\underline{\mathbf{y}})$ is the unique symmetric multilinear form associated to $\mathcal{F} (\y)$.
In particular, it follows that
\begin{equation}\label{XXXXX}
\mathcal{P}_{d_2}( \underline{\mathbf{y}} )
=
(-1)^{d_2} d_2!  \alpha \sum_{ 1 \leq  \mathbf{k} \leq n } \left(  \sum_{ 1 \leq   \mathbf{j}  \leq n } G_{ \mathbf{j}, \mathbf{k} }  x_{j_1} \cdots x_{j_{d_1}} \right)  y_{1, k_1} \cdots y_{d_2, k_{d_2}}
+
\alpha \mathfrak{g}_{\x, d_2}(\underline{\mathbf{y}}).
\end{equation}
Moreover, since $\mathfrak{g} (\x, \mathbf{y})$ is a polynomial of degree  in $\mathbf{x}$ strictly less than $d_1$, it follows that so are 
$\mathfrak{g}_{\x, d_2}(\underline{\mathbf{y}})$ and $\mathfrak{g}_{\x, d_2 + 1}(\underline{\mathbf{y}}, \mathbf{y}_{d_2 + 1})$.

We substitute the inequality (\ref{ineq 1-2}) into (\ref{ineq 1-1}), and we interchange the order of summation moving
the sum over $\mathbf{x}$ inside the sums over $\mathbf{y}_1, \ldots, \y_{d_2 + 1}$.
Coupled with an application of H\"{o}lder's inequality, this yields
\begin{equation}
\label{Hod1}
\begin{split}
|T(\alpha)|^{2^{d_1+ d_2 - 1}}
\leq~&
|\mathcal{Y}|^{2^{d_1 + d_2 - 1} - d_2 - 1} |\mathcal{X}|^{  2^{d_1 + d_2 - 1} - 2^{d_1 - 1}  }\\
&\times
\sum_{\mathbf{y}_1 \in \mathcal{Y}} \cdots  \sum_{\mathbf{y}_{d_2 + 1} \in \mathcal{Y}  } | \mathfrak{T}_{ \underline{ \y } } (\y_{d_2 + 1} ;  \alpha) |^{2^{d_1 - 1} },
\end{split}\end{equation}
where 
$$
\mathfrak{T}_{\underline{\y}} ( \y_{d_2 + 1}; \alpha)
= 
\sum_{ \x \in \mathcal{X}}
\psi \left(
  \alpha \mathfrak{g}_{\x, d_2 + 1} (  \underline{\mathbf{y}},   \mathbf{y}_{d_2 + 1})
- \alpha \sum_{ 1 \leq  \mathbf{j} \leq n}   \mathcal{H}_{\mathbf{j}}(\underline{\mathbf{y}}) x_{j_1} \cdots x_{j_{d_1}}
  - \alpha  \mathfrak{g}_{\mathbf{x}, d_2} (\underline{\mathbf{y}})
\right)
$$
and
$$
\mathcal{H}_{\mathbf{j}}(\underline{\mathbf{y}}) =
(-1)^{d_2} d_2!
\sum_{ 1 \leq  \mathbf{k} \leq n }  G_{\mathbf{j}, \mathbf{k} }  y_{1, k_1} \cdots y_{{d_2}, k_{d_2}}.
$$

We now apply the same differencing process as before to $|\mathfrak{T}_{\underline{\mathbf{y}}} (\y_{d_2 + 1}; \alpha)|^{2^{d_1 - 1}}$, but using 
$$
\mathcal{P}(\mathbf{x}) =   \alpha \mathfrak{g}_{\x, d_2 + 1} (  \underline{\mathbf{y}},   \mathbf{y}_{d_2 + 1})
- \alpha \sum_{ 1 \leq  \mathbf{j} \leq n}  \mathcal{H}_{\mathbf{j}}(\underline{\mathbf{y}})  x_{j_1} \cdots x_{j_{d_1}}
  - \alpha \mathfrak{g}_{\mathbf{x}, d_2} (\underline{\mathbf{y}})
$$
instead of (\ref{defnF=g}).
We define
$$
\mathcal{P}_{d_1} (\underline{\x}) =
\sum_{\epsilon_1 \in \{0, 1\} } \cdots \sum_{\epsilon_{d_1} \in \{0, 1\} }
(-1)^{\epsilon_1 + \cdots + \epsilon_{d_1} }
\mathcal{P} ( \epsilon_1 \mathbf{x}_1 + \cdots + \epsilon_{d_1} \mathbf{x}_{d_1} ),
$$
where $\underline{\x} = (\x_1, \ldots, \x_{d_1})$.  Since
$\mathfrak{g}_{\mathbf{x}, d_2} (\underline{\mathbf{y}})$ and  $\mathfrak{g}_{\x, d_2 + 1} (  \underline{\mathbf{y}},   \mathbf{y}_{d_2 + 1})$  are polynomials of degree 
strictly less than $d_1$ 
in $\x$,  we may argue similarly to the deduction of (\ref{XXXXX}) to obtain
$$
\mathcal{P}_{d_1}(\underline{\x}) =
(-1)^{d_1 + 1} d_1! \alpha
\sum_{ 1 \leq  \mathbf{j} \leq n} \mathcal{H}_{\mathbf{j}}(\underline{\mathbf{y}}) x_{1, j_1} \cdots x_{d_1, j_{d_1}}.
$$
 Similarly to (\ref{beforeCS}), an  application of Lemma \ref{Diff} yields
\begin{align*}
| \mathfrak{T}_{\underline{ \mathbf{y} }} ( \y_{d_2 + 1}; {\alpha}) |^{2^{d_1 - 1} }
&\leq  | \mathcal{X} |^{2^{d_1 - 1} - d_1} \sum_{\mathbf{x}_1 \in \mathcal{X}} \cdots
\sum_{\mathbf{x}_{d_1 - 1} \in \mathcal{X}}
\left| \sum_{\mathbf{x}_{d_1 } \in \mathcal{X}}  \psi ( \mathcal{P}_{d_1}( \underline{\x}) ) \right|
\\
&=
| \mathcal{X} |^{2^{d_1 - 1} - d_1} \sum_{\mathbf{x}_1 \in \mathcal{X}} \cdots
\sum_{\mathbf{x}_{d_1 - 1} \in \mathcal{X}} \sum_{\mathbf{x}_{d_1 } \in \mathcal{X}}  \psi ( \mathcal{P}_{d_1}( \underline{\x} ) ),
\end{align*}
where the final equality follows from 
(\ref{eq:laurel}), since  $\mathcal{P}_{d_1}( \underline{\x} )$ is linear in $\x_{d_1}$.
By substituting this estimate 
into (\ref{Hod1}) and recalling that $d = d_1 + d_2$, we obtain  the expression
\begin{align*}
|T (\alpha)|^{2^{ d - 1} }
&\leq
|\mathcal{X}|^{  2^{d - 1} - d_1  }
|\mathcal{Y}|^{  2^{d - 1} - d_2 - 1}
\sum_{\mathbf{y}_1 \in \mathcal{Y}} \cdots \sum_{\mathbf{y}_{d_2 + 1} \in \mathcal{Y}  }  \sum_{\mathbf{x}_1 \in \mathcal{X}} \cdots \sum_{\mathbf{x}_{d_1} \in \mathcal{X} }
\psi (  \mathcal{P}_{d_1}( \underline{\x} )  )
\\
&=
|\mathcal{X}|^{  2^{d - 1} - d_1  }
|\mathcal{Y}|^{  2^{d - 1} - d_2 }
\sum_{\mathbf{y}_1 \in \mathcal{Y}} \cdots \sum_{\mathbf{y}_{d_2 } \in \mathcal{Y}  }  \sum_{\mathbf{x}_1 \in \mathcal{X}} \cdots \sum_{\mathbf{x}_{d_1} \in \mathcal{X} }
\psi (  \mathcal{P}_{d_1}( \underline{\x} )  ),
\end{align*}
since  $\mathcal{P}_{d_1}(\underline{\x})$ is independent of $\y_{d_2+1}$.
Moreover, 
we can replace the phase function $\mathcal{P}_{d_1}( \underline{\x} )$ by
$$
 \frac{ (-1)^{d+1} }{d_1! d_2!} \mathcal{P}_{d_1}( \underline{\x} ) ,
$$
on making 
a change of variables.
Since this inequality will be important for us later, we record it as the following result.

\begin{lemma}
\label{LEMMMM}
With the notation in this section, we have
$$
|T (\alpha)|^{2^{ d - 1} } \leq
q^{  2^{d - 1} P_1 n  +  2^{d - 1}  P_2 n }  E(0)^{-1}  E(\alpha),
$$
where
$$
E(\alpha) =
\sum_{  |\underline{ \mathbf{x} }| < q^{P_1} } \sum_{  |\underline{ \mathbf{y} }| < q^{P_2} }
\psi ( \alpha \Gamma_G( \underline{\x} ; \underline{\y}  ) )
$$
and
$$
\Gamma_G( \underline{\x} ; \underline{\y}  ) =
\sum_{ 1 \leq  \mathbf{j} \leq n }  \sum_{ 1 \leq  \mathbf{k} \leq n} G_{\mathbf{j}, \mathbf{k} }
 x_{1, j_1} \cdots x_{d_1, j_{d_1}}  y_{1, k_1} \cdots y_{{d_2}, k_{d_2}}.
$$
\end{lemma}

Let $\underline{\mathbf{y}}' = (\mathbf{y}_1, \ldots, \mathbf{y}_{d_2 - 1})$.
Let $\mathbf{e}_{i}$ be the $i$-th unit vector in $\KK^n_\infty$ for each $1 \leq i \leq n$.
Given $Q_2 \in \ZZ_{\leq 0}$ we let
$$
N_2^{(t)}(Q_2; \alpha)
=
\# \left\{
\begin{array}{l}
 \underline{\x} \in  \FF_q[u]^{d_1 n}\\
  \underline{\mathbf{y}}' \in\FF_q[u]^{(d_2-1)n} 
  \end{array}
  :
\begin{array}{l}
 |\x_1|, \ldots, |\x_{d_1}| < q^{P_1}   \\
 |\y_1|,\ldots, |\y_t| < q^{ Q_2 +  P_2}   \\
  |\y_{t + 1}|,\ldots, |\y_{d_2-1}|  <  q^{ P_2} \\
 \| \alpha \Gamma_G(  \underline{\x} ; \underline{\mathbf{y}}',  \mathbf{e}_{i} ) \| < q^{  t Q_2- P_2} \\
 \textnormal{for all $1 \leq i \leq n$}
\end{array}
\right\},
$$
for each $0 \leq t \leq d_2-1$. Note that $N_2^{(0)}(Q_2; \alpha)$ is independent of $Q_2$.
Then it follows from \eqref{eq:laurel} that
\begin{equation}\begin{split}
\label{exp bound with count 1}
E(\alpha)
&=  \sum_{\mathbf{x}_{1} \in \mathcal{X}} \cdots  \sum_{\mathbf{x}_{d_1} \in \mathcal{X}}  \sum_{\mathbf{y}_{1} \in \mathcal{Y}} \cdots
\sum_{\mathbf{y}_{d_2 - 1} \in \mathcal{Y}} \prod_{i = 1}^n
\sum_{ |y_{d_2, i}| < q^{P_2}  }
\psi  (  \alpha \Gamma_G ( \underline{\x};  \underline{\mathbf{y}}', \mathbf{e}_{i} ) y_{d_2, i}  )
 \\
&=  |\mathcal{Y}| N_2^{(0)}(Q_2; \alpha).
\end{split}\end{equation}

We proceed by recording a version of Davenport's shrinking lemma over $\FF_q[u]$. The following is proved in \cite[Lemma 6.4]{BS1}, but phrased here in a manner that best fits our needs.

\begin{lemma}
\label{shrink2}
Let $\mathfrak{L}_1, \dots, \mathfrak{L}_N \in \KK_\infty[t_1, \dots, t_N]$ be symmetric linear forms given by
$$\mathfrak{L}_i(\t) = \gamma_{i,1} t_1 + \dots + \gamma_{i,N} t_N,$$
for $1 \leq i \leq N$; i.e.\ such that $\gamma_{i,j} = \gamma_{j,i}$ for $1 \leq i, j \leq N$.
For  $A \in \QQ_{\geq 0}$ and $Z \in \QQ_{\leq 0}$, let
$$
\mathcal{N}_{A}(Z)
= \#
\{ \t \in  \FF_q[u]^N: |t_i|< q^{A +  Z  }, \| \mathfrak{L}_i(\t) \| < q^{- A +  Z }  \text{ for all $1 \leq i \leq N$} \}.
$$
Then 
we have
$$
\frac{ \mathcal{N}_{A}(Z_2)}{ \mathcal{N}_{A} ( Z_1 )} \leq  q^{ N (  Z_2  -  Z_1 )  }
$$
for any choice of $Z_1, Z_2 \in \QQ_{\leq 0}$ satisfying 
$A - Z_2 \in \NN$, 
$Z_2 - Z_1 \in \mathbb{Z}_{\geq 0}$ and 
$
A \pm Z_i \in \mathbb{Z}$
for $i\in \{1,2\}$.
\end{lemma}

From this point onwards, we specialise our choice of $G(\x; \y)$ to satisfy the  following assumption.
\begin{hyp}[Symmetric linear forms hypothesis]\label{HYP}
There exist $C_0 \in \FF_q^*$ and a non-singular degree $d$ form $f \in \FF_q [x_1, \dots, x_n]$  such that
$$
G(\x ; \y) = C_0 \Gamma_f (\x, \dots, \x, \y, \dots, \y ),
$$
in the notation of \eqref{unique}, 
where $\x$ occupies the first $d_1$ slots and $\y$ the remaining $d_2$ slots in the expression on the right hand side.
\end{hyp}

Under this hypothesis, it follows from \eqref{eq:cherry}  and  \eqref{unique} that
\begin{align*}
G(\x ; \y) &=C_0  \sum_{i_1, \dots, i_d = 1}^n c_{i_1, \dots, i_d} x_{i_1} \dots x_{i_{ d_1 }} y_{i_{ d_1 + 1 }} \dots  y_{i_{ d }}
\\
&= C_0 \sum_{j_1, \dots, j_{d_1} = 1}^n  \sum_{k_1, \dots, k_{d_2} = 1}^n  c_{\mathbf{j}, \k} x_{j_1} \dots x_{j_{ d_1 }} y_{k_{ 1 }} \dots  y_{k_{ d_2 }}.
\end{align*}
Then
$$
\Gamma_G( \underline{\x}; \underline{\y})
=
C_0 \sum_{1 \leq \mathbf{j} \leq n}  \sum_{1 \leq \k \leq n}  c_{\mathbf{j}, \k} x_{1, j_1} \dots x_{d_1, j_{ d_1 }}
y_{1, k_{ 1 }} \dots  y_{d_2, k_{ d_2 }}.
$$
In particular, for any given $1 \leq t < d_2$ the system of linear forms
$$
\mathfrak{L}_i (\y) = \sum_{j=1}^n \Gamma_G( \underline{\x}; \y_{1}, \ldots, \y_{t-1}, \ee_j, \y_{t+1}, \ldots, \y_{d_2 - 1}, \ee_i ) y_j
$$
 is symmetric,
 for each $1 \leq i \leq n$.
 Similarly, given any $1 \leq t \leq d_1$ the system of linear forms
$$
\mathfrak{L}_i (\x) = \sum_{j=1}^n \Gamma_G(\x_{1}, \ldots, \x_{t-1}, \ee_j, \x_{t+1}, \ldots, \x_{d_1} ; \underline{\mathbf{y}}', \ee_i ) x_j
$$
 is symmetric, for each $1 \leq i \leq n$.

For  $\mathbf{Q}=(Q_1, Q_2) \in \ZZ_{\leq 0}^2$ we define
$$
N_1^{(t)}(\mathbf Q; \alpha)
=
\# \left\{
\begin{array}{l}
\underline{\x}\in \FF_q[u]^{d_1 n}\\
\underline{\mathbf{y}}'
 \in \FF_q[u]^{(d_2-1)n}
\end{array}
 :
\begin{array}{l}
 |\x_1|,\ldots, |\x_t| < q^{  Q_1 + P_1 } \\
 |\x_{t + 1}|,\ldots, |\x_{d_1} | < q^{P_1} \\
 |\y_1|, \ldots, |\y_{d_2-1}| < q^{Q_2 + P_2 }  \\
 \| \alpha \Gamma_G(  \underline{\x} ; \underline{\mathbf{y}}',  \mathbf{e}_{i} ) \| < q^{ t Q_1 + (d_2 - 1) Q_2 -  P_2 } \\
 \textnormal{for all $1 \leq i \leq n$}
\end{array}
\right\},
$$
for each $0 \leq t \leq d_1$. Note that $N_1^{(0)}(\mathbf Q; \alpha)$ is independent of $Q_1$.

\begin{lemma}\label{ineq 1-40}
Let 
$\mathbf{Q}=(Q_1, Q_2) \in \ZZ_{\leq 0}^2$. Then
$$N_2^{(0)}(Q_2; \alpha) \leq
q^{- d_1 Q_1 n} q^{-  (d_2-1) Q_2  n }  N_1^{(d_1)}(\bf Q ; \alpha).
$$
\end{lemma}
\begin{proof}
First we prove
$$
N_2^{(0)}(Q_2; \alpha) \leq q^{ - (d_2-1) Q_2 n  } N_2^{(d_2-1)}(Q_2; \alpha),
$$
for which it 
suffices to show that 
\begin{equation}\label{eq:tile}
N_2^{(t-1)}(Q_2; \alpha) \leq q^{- Q_2 n}  N_2^{(t)}(Q_2; \alpha),
\end{equation}
for each $1 \leq t \leq d_2-1$. This will be a straightforward consequence of Lemma \ref{shrink2}. To see this, we 
fix a choice of  $1 \leq t \leq d_2-1$ and a vector $\underline{\x}$ such that 
 $|\underline{\x}| <  q^{P_1}$. Moreover, we fix 
$\y_i \in \FF_q[u]^n$, for each $i \neq t$, satisfying
\begin{equation}
\label{cond++++++}
|\y_1|,\ldots, |\y_{t-1}| < q^{ Q_2 + P_2}
\quad
\textnormal{and}
\quad
 |\y_{t + 1}|,\ldots, |\y_{d_2-1}| < q^{P_2}.
\end{equation}
Let
\begin{align*}
\mathfrak{L}_{i} (\y)
&=
\alpha  \Gamma_G ( \underline{\x};  \y_{1}, \ldots, \y_{t-1}, \y , \y_{t + 1},  \ldots, \y_{d_2-1}, \mathbf{e}_{i})
\\
&=
\sum_{j = 1}^n \alpha  \Gamma_G ( \underline{\x};  \y_{1}, \ldots, \y_{t-1}, \mathbf{e}_j , \y_{t + 1},  \ldots, \y_{d_2-1}, \mathbf{e}_{i}) y_j,
\end{align*}
for each $1 \leq i \leq n$.
Given any $Z \in \QQ_{\leq 0}$ we write
$$
M_2(Z)  =
\# \left\{
 \y \in  \FF_q[u]^n:
\begin{array}{l}
  |\y | <  q^{ - Q_2 \frac{t - 1 }{2} + P_2 + Z  }  \\
  \|\mathfrak{L}_{i}(\y) \| <   q^{ Q_2 \frac{t - 1 }{2} - P_2 + Z  } \\
 \textnormal{for all $1 \leq i \leq n$}
\end{array}
\right\}.
$$
By applying Lemma \ref{shrink2} with
$$
Z_1 = Q_2 \frac{ t+1 }{2}, \quad Z_2 = Q_2 \frac{t - 1 }{2}
\quad \textnormal{and}
\quad A =  - Q_2 \frac{t - 1 }{2}   + P_2,
$$
we obtain
$
M_2(  Z_2  ) \leq q^{- Q_2 n} M_2(  Z_1  ).
$
Summing the inequality over all
$|\underline{\x}| <  q^{P_1}$ and $\y_{i}$ with $i \neq t$ such that (\ref{cond++++++}) holds,  we arrive at the inequality \eqref{eq:tile}.

Next we prove
$$
N_2^{(d_2-1)}(Q_2; \alpha) = N_1^{(0)}(\mathbf Q ; \alpha) \leq q^{ - d_1 Q_1 n  } N_1^{(d_1)}(
\mathbf 
Q; \alpha),
$$
which in turn follows by showing that
\begin{equation}\label{eq:tiler}
N_1^{(t-1)}(\mathbf  Q; \alpha) \leq  q^{- Q_1 n}  N_1^{(t)}(\mathbf Q; \alpha),
\end{equation}
for each $1 \leq t \leq d_1$.
Let us fix $1 \leq t \leq d_1$, $|\underline{\mathbf{y}}'| < q^{Q_2 + P_2}$, and 
and $\x_i \in \FF_q[u]^n$, for each $i \neq t$, satisfying
\begin{equation}
\label{cond+++000}
|\x_1|,\ldots, |\x_{t-1}| < q^{Q_1 + P_1}
\quad
\textnormal{and}
\quad
|\x_{t + 1}|,\ldots, |\x_{d_1}|  < q^{P_1}.
\end{equation}
Let
\begin{align*}
\mathfrak{L}_{i} (\x)
&=
\alpha  \Gamma_G (  \x_{1}, \ldots, \x_{t-1}, \x , \x_{t + 1},  \ldots, \x_{d_1}; \underline{\mathbf{y}}',   \mathbf{e}_{i})
\\
&= \sum_{j=1}^n  \alpha  \Gamma_G (  \x_{1}, \ldots, \x_{t-1}, \mathbf{e}_j , \x_{t + 1}, \ldots, \x_{d_1}; \underline{\mathbf{y}}',   \mathbf{e}_{i}) x_j
\end{align*}
for each $1 \leq i \leq n$.
Given any $Z \in \QQ_{\leq 0}$ we write
$$
M_1(Z)
=
\#
\left\{
 \x \in  \FF_q[u]^n:
\begin{array}{l}
   |\x| < q^{ - Q_1  \frac{t - 1 }{2}   - Q_2 \frac{d_2-1}{2} + \frac12 (P_1 + P_2) + Z    }    \\
   \|\mathfrak{L}_{i}(\x) \|  < q^{  Q_1  \frac{t - 1 }{2}   + Q_2 \frac{d_2-1}{2} - \frac12 (P_1 + P_2) + Z    }   \\
 \textnormal{for all $1 \leq i \leq n$}
\end{array}
\right\}.
$$
By applying Lemma \ref{shrink2} with
$$
Z_1 = Q_1  \frac{t + 1 }{2} +  Q_2 \frac{d_2-1}{2}  + \frac12 (P_1 - P_2), \quad
Z_2 = Q_1 \frac{t - 1 }{2} + Q_2 \frac{d_2-1}{2} + \frac12 (P_1 - P_2)
$$
and
$$
A = - Q_1  \frac{t - 1 }{2} -  Q_2 \frac{d_2-1}{2} + \frac12 (P_1 + P_2),
$$
we obtain
$
M_1( Z_2 )   \leq q^{- Q_1 n }  M_1( Z_1 ).
$
Summing the inequality over all
$|\underline{\mathbf{y}}'| < q^{ Q_2 + P_2 }$ and
$\x_{i}$ with $i \neq t$ such that (\ref{cond+++000}) holds,
we arrive at the inequality \eqref{eq:tiler}, which thereby completes the proof.
\end{proof}

Let us denote
$$
V^* = \left\{  (\x, \y) \in \AA^{2n}:  \frac{\partial G}{ \partial y_i } (\x; \y) = 0 \textnormal{ for all $1 \leq i \leq n$} \right\}
\quad
\textnormal{and}
\quad
\sigma_G = 2n - \dim V^*
$$
if $d_1 \geq 1$ and $d_2 \geq 2$,
$$
V^* = \left\{  \x  \in \AA^{n}:  \frac{\partial G}{ \partial y_i } (\x; \y) = 0 \textnormal{ for all $1 \leq i \leq n$} \right\}
\quad
\textnormal{and}
\quad
\sigma_G = n - \dim V^*
$$
if $d_1 \geq 1$ and $d_2 = 1$, and
$$
V^* = \left\{  \y  \in \AA^{n}:  \frac{\partial G}{ \partial y_i } (\y) = 0 \textnormal{ for all $1 \leq i \leq n$} \right\}
\quad
\textnormal{and}
\quad
\sigma_G = n - \dim V^*
$$
if $d_1 = 0$ and $d_2 \geq 2$.
We proceed by proving the following lower bound for $\sigma_G$.

\begin{lemma}
\label{lemma5.1}
Under Hypothesis \ref{HYP} we have
$
\sigma_G \geq n.
$
\end{lemma}

\begin{proof}
We handle the case in which $d_1 \geq 1$ and $d_2 \geq 2$, the remaining two cases being similar.  
It then follows from Hypothesis \ref{HYP} that 
$$
V^* = \{ (\x, \y) \in  \AA^{2n}:   \Gamma_f (\x, \dots, \x, \y, \dots, \y,  \ee_{i} ) = 0  \textnormal{ for all $1 \leq i \leq n$} \}.
$$
Intersection with the diagonal 
$\{ (\x, \y) \in  \AA^{2n}:   \x = \y  \}$ leads to the set of $\x\in \AA^n$ such that $\nabla f(\x)=\0$. 
The only such vector is $\x=\0$, since $f$ is assumed to be non-singular. But then it follows
from the affine dimension theorem that
$
\dim V^* \leq  n,
$
whence $\sigma_G \geq n$.
\end{proof}

We are now ready to record our key major and minor arc dichotomy concerning the exponential sum $E(\alpha)$ that appears in Lemma \ref{LEMMMM}. 

\begin{lemma}
\label{6.2}
Let $1 \leq J \leq P_1$ be an integer. Then one of the following two alternatives holds:

\textnormal{(i)} We have
$
| E (\alpha)|
\leq
(d-1)^{n} E(0)
q^{- \sigma_G J  }$.

\textnormal{(ii)}
There exist $g, a \in \FF_q[u]$, with $g$ monic, such that $\gcd(g, a) = 1$,
$$
0 < |g| \leq q^{  (d-1) (J-1) }
$$
and
$$
| g \alpha - a | < q^{ - d_1 P_1 - d_2 P_2 + (d - 1) J }.
$$
\end{lemma}

\begin{proof}
We define the affine variety
$$
\mathcal{Z} = \{   (\underline{\mathbf{x}}, \underline{\mathbf{y}}' ) \in \mathbb{A}^{(d - 1) n }  :
\Gamma_G(  \underline{\mathbf{x}};  \underline{\mathbf{y}}', \mathbf{e}_{i} ) = 0   \textnormal{ for all $1 \leq i \leq n$}  \}
$$
and let
$$
\mathcal{M}(\mathcal{Z})
=
\notag
\left\{
(\underline{\mathbf{x}}, \underline{\mathbf{y}}') \in \mathcal{Z} \cap \mathbb{F}_q[u]^{(d - 1) n}  :
\begin{array}{l}
|\x_{1}|, \ldots, |\x_{d_1}| < q^{ Q_1 + P_1 }  \\
|\y_{1}|, \ldots, |\y_{d_2-1}|  < q^{ Q_2 + P_2 }
\end{array}{}
\right\},
$$
for integers $Q_1,Q_2\leq 0$.
In this proof, we choose $Q_1 = J - P_1$ and $Q_2 = J - P_2$.

Suppose every point counted by
$N_1^{(d_1)}(\mathbf Q; \alpha)$ is contained in $\mathcal{M}(\mathcal{Z})$. Then 
\begin{equation}
\label{ineq 4 10}
N_1^{(d_1)}(\mathbf Q; \alpha)
\leq
\# \mathcal{M}(\mathcal{Z})
\leq (d-1)^n q^{J \dim \mathcal{Z}} ,
\end{equation}
by  Lemma \ref{dimbdd}. It therefore follows from (\ref{exp bound with count 1}),
(\ref{ineq 4 10}) 
and Lemma \ref{ineq 1-40} that
\begin{equation}\begin{split}
\label{6.23}
| E (\alpha)|
&\leq  (d-1)^n q^{ P_2 n}
q^{- d_1 (J - P_1)  n} q^{- (d_2 - 1) (J- P_2) n } q^{J \dim \mathcal{Z}}
\\
&=
(d-1)^n  q^{d_1  P_1  n +  d_2  P_2 n } q^{- d_1 J   n} q^{- (d_2 - 1) J  n } q^{J \dim \mathcal{Z}}
\\
&=
(d-1)^n  E(0)  q^{- d_1 J   n} q^{- (d_2 - 1) J  n } q^{J \dim \mathcal{Z}}.
\end{split}\end{equation}
Let us denote
$$
\mathcal{D} = \{  (\underline{\mathbf{x}}, \underline{\mathbf{y}}') \in \mathbb{A}^{(d - 1 )n }:  \mathbf{x}_1 = \cdots = \mathbf{x}_{d_1} \text{ and }
\mathbf{y}_1 = \cdots = \mathbf{y}_{d_2-1} \}
$$
if $d_1 \geq 1$ and $d_2 \geq 1$, and
$$
\mathcal{D} = \{  \underline{\mathbf{y}}' \in \mathbb{A}^{(d - 1 )n }:  \mathbf{y}_1 = \cdots = \mathbf{y}_{d_2 - 1} \}
$$
if $d_1 = 0$ and $d_2 \geq 2$. Then it follows from the affine dimension theorem that
\begin{align*}
\dim V^{*} &= \dim ( \mathcal{Z} \cap \mathcal{D} )
\\
&\geq  \dim \mathcal{Z}
+ \dim \mathcal{D} - (d - 1) n
\\
&=
\begin{cases}
\dim \mathcal{Z}  -  (d_1 - 1) n -   (d_2 - 2) n  & \mbox{if $d_1 \geq 1$ and $d_2 \geq 2$,}  \\
\dim \mathcal{Z}  -  (d_1 - 1) n     & \mbox{if $d_1 \geq 1$ and $d_2 = 1$,}  \\
\dim \mathcal{Z}  -  (d_2 - 2) n    & \mbox{if $d_1 = 0$ and $d_2 \geq 2$.}
\end{cases}
\end{align*}
With this inequality and the definition of $\sigma_G$, (\ref{6.23}) becomes
$$
| E (\alpha)|
\leq
(d-1)^{n} E(0) q^{- \sigma_G J },
$$
which is precisely the estimate in alternative (i).

On the other hand, suppose there exists $(\underline{\mathbf{x}}, \underline{\mathbf{y}}')$ counted by
$N_1^{(d_1)}(\mathbf Q; \alpha)$ which is not contained in $\mathcal{M}(\mathcal{Z})$.
Then there exists $1 \leq i_0 \leq n$ such that
$
 \Gamma_G(\underline{\mathbf{x}};  \underline{\mathbf{y}}', \mathbf{e}_{i_0})$ is a non-zero element of $\FF_q[u]$.
Let us write
$$
\alpha \Gamma_G(\underline{\mathbf{x}};  \underline{\mathbf{y}}', \mathbf{e}_{i_0}) = a + \xi,
$$
where $a \in \FF_q[u]$ and
$$
|\xi | <  q^{  d_1 Q_1 + (d_2 - 1) Q_2 - P_2 } = q^{ - d_1 P_1 - d_2 P_2 + (d_1 + d_2 - 1) J }  = q^{ - d_1 P_1 - d_2 P_2 + (d - 1) J }.
$$
Since
$$
1 \leq |\Gamma_G(\underline{\mathbf{x}};  \underline{\mathbf{y}}', \mathbf{e}_{i_0})| \leq q^{ (d_1 + d_2 - 1) (J-1) }  = q^{ (d - 1) (J - 1) },
$$
alternative (ii) follows on taking  taking out a common factor from $\Gamma_G(\underline{\mathbf{x}};  \underline{\mathbf{y}}', \mathbf{e}_{i_0})$ and $a$ if necessary. \end{proof}

\subsection{Mean value estimate}
Let us define
$$
\mathfrak{M}(J)=
\hspace{-0.2cm}
\bigcup_{\substack{g\in \FF_q[u] \text{ monic}\\
0<|g|\leq q^{(d-1)(J-1)}
}}
\bigcup_{\substack{a \in \FF_q[u]\\
|a|<|g|\\ \gcd(a,g)=1}} \left\{\alpha \in \TT: \left|\alpha - \frac{ a}{g}\right|<\frac{q^{-d_1
P_1 - d_2 P_2 + (d-1) J }}{|g|}\right\},
$$
which  is precisely the set that appears in alternative (ii) of Lemma \ref{6.2}.
For $\varrho > 0$ and $J \in \NN$ we let
$$
I_J (\varrho) =
  \int_{ \mathfrak{M}(J+1) \setminus  \mathfrak{M}(J)} |E(\alpha)|^\varrho   \d \alpha.
$$
We note that $I_J(\varrho) = 0$ if
\begin{equation}
\label{Diri}
J \geq \frac{ d_1 P_1 + d_2 P_2  + d - 1 }{ 2( d -  1)},
\end{equation}
since it follows from Dirichlet's approximation theorem 
\cite[Lemma 3]{kubota}  that $\mathfrak{M}(J) = \TT$
in this case.
\begin{lemma}\label{lem:IJ-large}
Let $J \in \NN$.
Suppose
$$
\sigma_G  \varrho > 2 (d-1)
$$
and
$$
P_1 \geq \frac{ d_1 P_1 + d_2 P_2  + d - 1 }{ 2( d -  1)}.
$$
Then
$$
I_J  (\varrho) \leq (d-1)^{ n \varrho} E(0)^{\varrho} q^{  - d_1 P_1 - d_2 P_2 + d-1 - \delta J},
$$
where $\delta = \sigma_G \varrho - 2 (d-1)$.
\end{lemma}

\begin{proof}
First suppose that $J\geq P_1$. Then under our assumption on $P_1$ it follows 
from (\ref{Diri}) that $I_J(\varrho) = 0$ and we are done. Suppose now that  $ 1 \leq J < P_1$.
Given $\alpha \in \mathfrak{M}(J+1) \setminus  \mathfrak{M}(J)$, it follows from Lemma \ref{6.2} that
$$
|E(\alpha)|^{\varrho} \leq  (d-1)^{ n \varrho}  E(0)^{\varrho} q^{ - \sigma_G  \varrho J}.
$$
It is clear that
$$
 \int_{ \mathfrak{M}(J+1) \setminus  \mathfrak{M}(J)}    \d \alpha \leq
  \int_{ \mathfrak{M}(J+1)}   \d \alpha \leq q^{ - d_1 P_1 - d_2 P_2 + (d-1) (J+1) + (d-1) J}.
$$
Therefore, we obtain
\begin{align*}
I_J  (\varrho)
&\leq
q^{ - d_1 P_1 - d_2 P_2 + (d-1) (J+1)   + (d-1) J}
(d-1)^{ n \varrho}  E(0)^{\varrho} q^{ - \sigma_G  \varrho J}
\\
&\leq
(d-1)^{ n \varrho}
E(0)^{\varrho} q^{   - d_1 P_1 - d_2 P_2 + d - 1}
q^{  ( 2  (d-1)  - \sigma_G \varrho) J}
\\
&=
(d-1)^{ n \varrho} E(0)^{\varrho} q^{ - d_1 P_1 - d_2 P_2 + d-1 - \delta J},
\end{align*}
where $\delta = \sigma_G \varrho - 2 (d-1)$.
\end{proof}

\begin{proposition}
\label{MVT}
Suppose
$$
\sigma_G  \varrho > 2 (d-1)
$$
and
$$
P_1 \geq \frac{ d_1 P_1 + d_2 P_2  + d - 1 }{ 2( d -  1)}.
$$
Then there exits $\delta > 0$ such that
$$
\int_{\TT} |E(\alpha)|^\varrho   \d \alpha \leq  E(0)^{\varrho} q^{ - d_1 P_1 - d_2 P_2 + d - 1 }
\left( 1 + (d-1)^{ n \varrho}  \frac{q^{- \delta} }{1 - q^{- \delta}} \right).
$$
\end{proposition}
\begin{proof}
First we have the trivial bound
$$
\int_{\mathfrak{M}(1)} |E(\alpha)|^\varrho   \d \alpha \leq   E(0)^{\varrho} q^{ - d_1 P_1 - d_2 P_2 + d - 1 }.
$$
Therefore, by combining this estimate with Lemma \ref{lem:IJ-large}, we obtain
\begin{align*}
\int_{\TT} |E(\alpha)|^\varrho   \d \alpha &= \int_{\mathfrak{M}(1)} |E(\alpha)|^\varrho   \d \alpha  + \sum_{J \geq 1} I_J  (\varrho)
\\
&\leq
 E(0)^{\varrho}  q^{ - d_1 P_1 - d_2 P_2 + d - 1 } \left( 1 + (d-1)^{ n \varrho}  \frac{q^{- \delta} }{1 - q^{- \delta}} \right),
\end{align*}
as required.
\end{proof}

\section{Proof of Theorem \ref{t:2}}
\label{sec:proof}
We recall that $\FF_q$ is a finite field of characteristic $p > d$. We have 
$$
N(e) = \#\left\{
 \underline{ \mathbf{t} } = (\t_0, \dots, \t_e) \in \FF_q[u]^{n (e + 1) } :
\begin{array}{l}
|\t_s|  < q^{s + 1} \text{ for all $0 \leq s \leq e$}  \\
F_j(  \underline{ \t }  ) = 0 \text{ for all $0 \leq j \leq d e$}
\end{array}
\right\}
$$
in  \eqref{big*},
where $F_j(  \underline{ \t }  )$  are defined in \eqref{eq:ash}. 
It follows from \eqref{unique} that 
$$
F_j( \underline{ \t }  ) =
\sum_{ \substack{ 0 \leq s_1, \dots, s_d \leq e \\ s_1 + \dots + s_d =  j }} \Gamma_f (\t_{s_1}, \dots, \t_{s_d} ),
$$
for each $0 \leq j \leq d e$. 
On appealing to \eqref{eq:ortho}, we may now write
$$
N(e) =
\int_{\TT^{de+1}}S(\bal)\d \bal,
$$
where
\begin{equation}
\label{SSS}
S(\bal) = \sum_{  \underline{\t} \in \mathfrak{U}  }  \psi \left( \sum_{k =  0 }^{d e }  \alpha_k F_k(  \underline{\t}   )
\right)
\end{equation}	
and
$$
\mathfrak{U} = \{  \underline{\t} \in   \FF_q[u]^{ n (e+1) }:   |\t_s| < q^{s + 1} \textnormal{ for all  $0 \leq s \leq e$} \}.
$$
Let $j\in \{0 ,\dots,de\}$. Then any such $j$ can be written  $j = (\ell - 1) d + r$ with $0 \leq \ell \leq e$ and  $1 \leq r \leq d$.

Suppose first  that $r<d$ in this representation. 
If there is a term in $\sum_{k =  0 }^{d e }  \alpha_{k} F_{k}(  \underline{\t})$ whose
degree in $\t_{\ell - 1}$ is  greater than or equal to $d - r$,
and whose degree in $\t_{\ell}$ is greater than or equal to $r$, then it can only come from terms
$\Gamma_f (\t_{s_1}, \dots, \t_{s_d} )$ with precisely $d - r$ of the indices $s_1, \dots, s_d$ equal to  $\ell - 1$,
and the remaining $r$ indices equal to $\ell$. This restriction on the $s_1, \dots, s_d$ is equivalent to
requiring that $\ell - 1 \leq s_1, \dots, s_d \leq \ell$ and $s_1 + \dots + s_d =  (\ell - 1) d + r$. 
Letting $I_j = \{ \ell - 1, \ell \}$, 
 we therefore  deduce that 
 $$
\sum_{k =  0 }^{d e }  \alpha_{k} F_{k}(  \underline{\t})  =  \alpha_j G_j(\t_{\ell - 1} ;  \t_{\ell} ) + \mathfrak{f}_j (\underline{\t} ) + \mathfrak{g}_j (\underline{\t}),
$$
where
$$
G_j(\t_{\ell - 1} ;  \t_{\ell} ) = \sum_{ \substack{ s_1, \dots, s_d \in I_j  \\ s_1 + \dots + s_d =  (\ell - 1) d + r }} \Gamma_f (\t_{s_1}, \dots, \t_{s_d} )
$$
is bihomogeneous of bidegree $(d-r, r)$, and
$\mathfrak{f}_j (\underline{\t} )$ and $\mathfrak{g}_j (\underline{\t})$
are such that every term of $\mathfrak{f}_j (\underline{\t})$  has degree in $\t_{\ell}$ strictly less than $r$ and every term of
$\mathfrak{g}_j (\underline{\t}) $ has degree in $\t_{\ell - 1}$  strictly less than $d - r$.
It will be convenient to observe that
\begin{equation}
\label{j1}
- (d - r) \ell -  r (\ell + 1) + d - 1 = - (\ell - 1) d -  r  - 1 = - j - 1.
\end{equation}

Suppose next that  $r = d$ and write $I_j=\{\ell\}$.
Then there is a decomposition
$$
\sum_{k =  0 }^{d e }  \alpha_k F_k(  \underline{\t})  =  \alpha_j G_j(\t_{\ell} ) + \mathfrak{f}_j (\underline{\t}),
$$
where
$$
G_j (\t_{\ell} ) =  \Gamma_f (\t_{\ell}, \dots, \t_{\ell} )
$$
is homogeneous of degree $d$ and $\mathfrak{f}_j (\underline{\t})$ is of degree in $\t_\ell$ strictly less than $d$.
This time it will be convenient to observe that
\begin{equation}
\label{j2}
- d (\ell + 1) + d - 1 = - \ell d - 1 = - j - 1.
\end{equation}

For each $j\in \{0,\dots, de\}$, it is clear that we are in the setting of Section \ref{sec:BHMG}
and that  Hypothesis \ref{HYP} is satisfied.
Let us put 
$$
\mathfrak{U}_{j} = \{  (\t_s)_{s \in I_j}  \in   \FF_q[u]^{ n \# I_j }:   |\t_{s}| < q^{s + 1} \textnormal{ for all  $s \in I_j$}  \}
$$
and
$$
\mathfrak{V}_{j} =  \left\{  \underline{ \widetilde{\t} } = (\t_s)_{s \notin I_j}  \in   \FF_q[u]^{ n (e + 1-  \# I_j) }:   |\t_{s}| < q^{s + 1} \textnormal{ for all  $s \not \in I_j$}  \right\}.
$$
We can write
$$
S(\bal) = \sum_{  \widetilde{ \underline{\t} } \in \mathfrak{V}_{j}  } T_{j}(\widetilde{ \underline{\t} }; \alpha_{j} ),
$$
where
$$
T_{j}(\widetilde{ \underline{\t} }; \alpha_{j})
=
  \sum_{ (\t_{\ell - 1}, \t_{\ell} ) \in  \mathfrak{U}_{j}  } \psi(  \alpha_j G_j(\t_{\ell - 1} ;  \t_{\ell} ) + \mathfrak{f}_j (\underline{\t} ) + \mathfrak{g}_j (\underline{\t}) ) ,
$$
if $j = (\ell - 1) d + r$ and $1 \leq r < d$, and where 
$$
T_{j}(\widetilde{ \underline{\t} }; \alpha_{j})
=
  \sum_{ \t_{\ell} \in  \mathfrak{U}_{j}  } \psi(  \alpha_j G_j(\t_{\ell} ) + \mathfrak{f}_j (\underline{\t} )  ) ,
  $$
  if $j = \ell d $.
An application of Lemma \ref{LEMMMM} now yields
\begin{align*}
|T_{j}(\widetilde{ \underline{\t} }; \alpha_{j}) |^{2^{d-1}} \leq   |\mathfrak{U}_{j}|^{2^{d-1}}
E_{j}(0)^{ - 1 }  |E_{j}(\alpha_{j})|,
\end{align*}
where $E_{j}(\alpha_{j})$ is as in the statement of the lemma. In particular, since $E_{j}(\alpha_{j})$
is independent of $\widetilde{ \underline{\t} }$, we may further deduce that
\begin{align*}
|S(\bal)|^{2^{d-1}} &\leq  |\mathfrak{V}_{j}|^{2^{d-1}}   |\mathfrak{U}_{j}|^{2^{d-1}}
E_{j}(0)^{ - 1 }  |E_{j}(\alpha_{j})|
\\
&=
|\mathfrak{U}|^{2^{d-1}}
E_{j}(0)^{ - 1 }  |E_{j}(\alpha_{j})|.
\end{align*}
Therefore
\begin{equation}
\begin{split}
\label{decomp}
|S(\bal)|^{2^{d-1}}
=
\prod_{j = 0}^{de}  |S(\bal)|^{ \frac{2^{d-1}}{de + 1} }
&\leq
\prod_{j = 0}^{de}
|\mathfrak{U}|^{  \frac{2^{d-1}}{de + 1} }   E_{j} (0)^{  -  \frac{1}{de + 1} }
| E_{j} (\alpha_{j})  |^{  \frac{1}{de + 1} }
\\
&=
\mathfrak{E}^{2^{d-1}  }  |\mathfrak{U}|^{2^{d-1}  } \prod_{j = 0}^{de}    | E_j (\alpha_j)  |^{  \frac{1}{de + 1}  },
\end{split}
\end{equation}
where
\begin{equation}\label{eq:slab}
\mathfrak{E} = \prod_{ j = 0}^{de}
E_{j} (0) ^{  -  \frac{1}{ (de + 1) 2^{d-1}} }.
\end{equation}

\begin{lemma}\label{lem:slab}
Suppose $n > 2^{d} (d-1) (de + 1)$. For each $0 \leq  (\ell - 1) d + r \leq de$ with $0 \leq \ell \leq e$ and  $1 \leq r \leq d$, we have
$$
\frac{ \sigma_{G_{ (\ell - 1) d + r} }  }{ (de + 1) 2^{d-1}}   > 2 (d-1)
$$
and
$$
 \ell  
 \geq 
 \begin{cases}
   \frac{ (d-r) \ell  + r (\ell + 1) - d + 1}{ 2(d-1) } & \text{ if $1 \leq r < d$},\\
 \frac{ d (\ell + 1)  - d + 1}{ 2 (d-1) }-1 & \text{ if $r = d$}.
 \end{cases}
$$
\end{lemma}
\begin{proof}
By Lemma \ref{lemma5.1} we have $\sigma_{G_j} \geq n$ for each $0 \leq j \leq de$. Therefore,
the first statement follows immediately from the hypothesis $n > 2^{d} (d-1)(de + 1)$. For the second statement, when $1\leq r<d$, 
it is easy to see that
$$
\ell  \geq    \frac{  d \ell  }{ 2(d-1) } \geq    \frac{ (d-r) \ell  + r (\ell + 1) - d + 1}{ 2(d-1) }.
$$
Similarly, 
$$
\ell + 1  \geq \frac{ d (\ell + 1) }{ 2 (d-1) } > \frac{ d (\ell + 1)  - d + 1}{ 2 (d-1) } , 
$$
when  $r=d$.
\end{proof}

Finally, on returning to   (\ref{decomp}), we see that 
\begin{align*}
\int_{\TT^{de+1}} |S(\bal)| \d \bal
&\leq
\mathfrak{E} |\mathfrak{U}| \prod_{j=0}^{de}  \int_{\TT} |E_j(\alpha_j)|^{ \frac{1}{ (de + 1) 2^{d-1} } } \d \alpha_j,
\end{align*}
where 
$\mathfrak{E}$ is given by 
\eqref{eq:slab}.
On  writing $j=(\ell - 1) d + r $ with $0 \leq \ell \leq e$ and  $1 \leq r \leq d$, 
we would like to apply  Proposition \ref{MVT} 
to estimate the remaining integral. We wish to apply this result with
with $\rho=\frac{1}{(de+1)2^{d-1}}$, together with the choices 
$$
(P_1,P_2)=
\begin{cases}
(\ell,\ell+1) & \text{ if $r<d$,}\\
(\ell+1,\ell+1) & \text{ if $r=d$.}
\end{cases}
$$
The required  lower bounds on $\sigma_{G_j}$ and $P_1$ now  follow from 
Lemma \ref{lem:slab}.
Hence it follows from (\ref{j1}), (\ref{j2}) and 
\eqref{eq:slab}
that
\begin{align*}
\int_{\TT^{de+1}} |S(\bal)| \d \bal
&\leq
\mathfrak{E} |\mathfrak{U}|
\prod_{ j = 0 }^{de}  E_{j}(0)^{ \frac{1}{ (de + 1) 2^{d-1} } }  q^{- j - 1 } (1 + O( q^{- \delta} ))
\\
&=
|\mathfrak{U}|
\prod_{ j = 0 }^{de}  q^{ - j - 1 } (1 + O( q^{- \delta} ))
\\
&=
|\mathfrak{U}| q^{  -  \sum_{j = 0}^{de} (j + 1) } (1 + O( q^{- \delta} ))
\\
&=
q^{ \widehat{\mu}(e) }  (1 + O( q^{- \delta} )),
\end{align*}
where 
$\widehat{\mu}(e)$ is defined in  \eqref{eq:mu'} and
the implicit constant depends only on $n, d, e$ and $\delta$.
This establishes (\ref{eq:goat}) as required to complete the proof of Theorem \ref{t:2}.

\begin{remark}
\label{sec:finalrem}
We now give a non-rigorous explanation of why our choice of the bihomogeneous structure
is essentially optimal, in terms of the dimension of the Birch singular locus (\ref{sing}), among all possible choices of different Weyl differencing processes.
Let us choose subsets $\mathfrak{I}_i \subset \{ 0, \dots, e \}$ for each $1 \leq i \leq d$.
We can modify the selection of the differencing process in Section \ref{sec:DIFF}
by differencing with respect to $\t_{s}$ with $s \in \mathfrak{I}_1$ in the first round, then with $s \in \mathfrak{I}_2$ in the second round,
and so on, with the effect that we only pick out the terms whose monomials are of the form
$$
t_{s_1, i_1} \dots t_{s_d, i_d},
$$
with 
$
(s_{\sigma(1)}, \dots, s_{ \sigma(d)}) \in \mathfrak{I}_1 \times \dots \times \mathfrak{I}_d
$
for some permutation $\{\sigma(1), \dots, \sigma(d)\}=\{1,\dots,d\}$.
The usual Weyl differencing process corresponds to taking $\mathfrak{I}_i = \{ 0, \dots, e \}$ for each $1 \leq i \leq d$,
while when $r<d$ our choice in Section \ref{sec:proof} corresponds to 
$$
\mathfrak{I}_i 
=
\begin{cases}
 \{ \ell \} & \text{ for $1 \leq i \leq d-r$,}\\
  \{ \ell -  1  \} & \text{ for $d - r < i \leq d$}.
  \end{cases}
  $$
We may observe that every $\alpha_j$ remaining in the resulting exponential sum corresponds to
\begin{equation}
\label{set}
j \in \{ s_1 + \dots  + s_d : s_1 \in \mathfrak{I}_1, \dots, s_d \in \mathfrak{I}_d \}.
\end{equation}
In particular, if we denote by $j_0$ the smallest such $j$, then
$$
j_0 = s_1' + \dots + s_d',
$$
where $s_i' = \min_{s \in  \mathfrak{I}_i} s$ for each $1 \leq i \leq d$, is the unique representation of $j_0$ in the set (\ref{set}).
This means that in the Weyl differencing argument behind Lemma \ref{6.2}, after intersecting with the relevant diagonal,
the row corresponding to $j_0$ in the Jacobian of the resulting system of forms corresponds to the first partial derivatives of
$$
c \Gamma_{f} (\t_{s'_1}, \dots, \t_{s'_d}),
$$
for some $c \in \FF_q$. 
By multilinearity, it is readily seen that the Birch singular locus (\ref{sing}) of any system of forms that includes $c \Gamma_{f} (\t_{s'_1}, \dots, \t_{s'_d})$
has codimension at most $n$ in the ambient space. Thus it appears that 
the Weyl differencing process results in the codimension of the Birch singular locus being at most $n$, no matter how we choose the sets 
 $\mathfrak{I}_1,\dots,\mathfrak{I}_d$.
\end{remark}

\end{document}